\documentclass[11pt,a4paper,reqno]{amsart}

\usepackage[T1]{fontenc}
\usepackage[utf8]{inputenc}
\usepackage{xspace,
		amsfonts}
\usepackage{amsmath} 
\usepackage{amssymb}
\usepackage{graphicx,
		bbm,
		tikz,subfig,soul,
	}
\usepackage{amsthm}
\usepackage{mathtools}
\usepackage{enumitem}
\usepackage{verbatim}
\usepackage[autostyle]{csquotes}

\usepackage{cite}

\usepackage{geometry}
\newtheorem{theorem}{Theorem}[section]
\newtheorem{lemma}[theorem]{Lemma}
\newtheorem{proposition}[theorem]{Proposition}
\newtheorem{corollary}[theorem]{Corollary}

\newtheorem*{remark*}{Remark}

\DeclarePairedDelimiter\abs{\lvert}{\rvert}%

\DeclarePairedDelimiter\norm{\lVert}{\rVert}%

\newcommand{\deb}{\rightharpoonup}

\newcommand{\sgn}{\text{sgn}}
\newcommand{\R}{\mathbb{R}}
\newcommand{\G}{\mathcal{G}}

\newcommand{\nlc}[3]{\norm{#1}_{L^{#2}(#3)}^{#2}}
\newcommand{\F}{\mathcal{F}}
\newcommand{\N}{\mathbb{N}}

\newcommand{\Eps}{\mathcal{E}}

\newcommand{\I}{\mathcal{I}}
\newcommand{\HH}{\mathcal{H}}

\newcommand{\HmuS}{H_\mu^1(S_N)}
\newcommand{\sech}{\text{sech}}
\newcommand{\f}[2]{\frac{#1}{#2}}
\newcommand{\uLp}{\|u\|_p}
\newcommand\vv{\textsc{v}}
\newcommand{\M}{\mathcal{M}}

\begin{document}

\title[Competing nonlinearities in NLSE on star graphs]{Competing nonlinearities in NLS equations as source of threshold phenomena on star graphs}

\author[R. Adami]{Riccardo Adami$^1$}
\address{$^1$Politecnico di Torino, Dipartimento di Scienze Matematiche ``G.L. Lagrange'', Corso Duca degli Abruzzi, 24, 10129, Torino, Italy}
\email{riccardo.adami@polito.it}
\author[F. Boni]{Filippo Boni$^{2,3}$}
\address{$^2$Politecnico di Torino, Dipartimento di Scienze Matematiche ``G.L. Lagrange'', Corso Duca degli Abruzzi, 24, 10129, Torino, Italy}
\email{filippo.boni@polito.it}
\address{$^3$Universit\`a degli Studi di Torino, Dipartimento di Matematica ``G. Peano'', Via Carlo Alberto, 10, 10123, Torino, Italy}
\email{filippo.boni@unito.it}
\author[S. Dovetta]{Simone Dovetta$^4$}
\address{$^4$Istituto di Matematica Applicata e Tecnologie Informatiche "E. Magenes", Via Adolfo Ferrata, 1, 27100, Pavia, Italy}
\email{simone.dovetta@imati.cnr.it}

\date{\today}

\begin{abstract}
We investigate the existence of ground states for the nonlinear Schr\"odinger Equation on star graphs with two subcritical focusing nonlinear terms: a standard power nonlinearity, and a delta-type nonlinearity located at the vertex. We find that if the standard nonlinearity is stronger than the pointwise one, then ground states exist for small mass only. On the contrary, if the pointwise nonlinearity prevails, then ground states exist for large mass only. All ground states are radial, in the sense that their restriction to each half--line is always the same function, and coincides with a soliton tail.
Finally, if the two nonlinearities are of the same size, then the existence of ground states is insensitive to the value of the mass, and holds only on graphs with a small number of half--lines. 

Furthermore, we establish the orbital stability of the branch of radial stationary states to which the ground states belong, also in the mass regimes in which there is no ground state. 
\end{abstract}

\maketitle

\vspace{-.5cm}
\smallskip


\section{Introduction}

In this paper we investigate existence and uniqueness of ground states for the NLS energy functional
\begin{equation}
\label{Fpq1}
F_{p,q}(u)=\frac{1}{2}\int_{S_N}\abs{u'}^2\,dx-\frac{1}{p}\int_{S_N}\abs{u}^p\,dx-\frac{1}{q}\abs{u(0)}^q,
\end{equation}
with the mass constraint
\begin{equation}
	\label{mass}
	\int_{S_N}|u|^2\,dx=\mu\,,
\end{equation}
where $S_N$ is the star graph made of $N$ half--lines glued together at their common origin, which in the following will be denoted by $\vv$ or, alternatively, by $0$.

Each half--line $\mathcal{H}_i$ of $S_N$ is provided with a coordinate $x_i:\mathcal{H}_i\to[0,+\infty)$, so that $x_{i}=0$ corresponds to the origin of $\mathcal{H}_i$, for every $i=1,\,\dots,\,N$. A function $u$ on $S_N$ is given by the collection of its restrictions $u_i$ to each half--line $\mathcal{H}_i$. We then define $L^p(S_N)$ as the space of all functions $u$ on $S_N$ such that, for every $i$, $u_i\in L^p(\mathcal{H}_i)$ with respect to the standard Lebesgue measure on $\mathcal{H}_i$. We define
\[
\|u\|_{L^p(S_N)}^p:=\sum_{i=1}^N\|u_i\|_{L^p(\mathcal{H}_i)}^p\,.
\]
Similarly, we set $H^1(S_N)$ to be the space of all functions $u$ on $S_N$ that are continuous on the graph, in particular at $\vv$ , and such that $u_i\in H^1(\mathcal{H}_i)$ for every $i$, endowed with the norm
\[
\|u\|_{H^1(S_N)}^2:=\sum_{i=1}^N\|u_i\|_{H^1(\mathcal{H}_i)}^2\,.
\]
Introducing the notation
\[
\HmuS:=\big\{u\in H^1(S_N)\,:\,\|u\|_{L^2(S_N)}^2=\mu\big\}
\]
for the mass constrained space, and
\begin{equation}
\label{infFpq}
\F_{p,q}(\mu):=\inf_{v\in H^1_\mu(S_N)}F_{p,q}(v)
\end{equation}
for the ground state energy level of \eqref{Fpq1} in $\HmuS$, we define a ground state of \eqref{Fpq1} at mass $\mu$ as a minimizer of the energy among functions with mass $\mu$, i.e. $u\in\HmuS$ such that $F_{p,q}(u)=\F_{p,q}(\mu)$. 

The aim of the paper is thus to understand whether ground states exist and are unique in the $L^2$ subcritical regime for both nonlinearities, namely $p\in(2,6)$ and $q\in(2,4)$. 

\smallskip

Since their first appearance as a model for valence electrons in naphthalene molecules \cite{RS}, the study of dynamics on metric graphs (or networks) has grown through the decades to become a prominent line of research. To date, models involving Schr\"odinger operators have been gathering a considerable interest, both in the linear setting (see for instance \cite{BL,BKKM,EFK,kennedy,KKMM} as well as the monograph \cite{BK} and references therein) and in the nonlinear case \cite{abr,ADST,ASTcpde,AST,AST-CMP,BMP,CF,CDS,D-jde,D-nodea,DGMP,dst_aim,DSTjlms,GKP,KP-JDE,KP-JPA,N,NP,NPS,Pa,PS,PSV}. Recent investigations are now available also for the KdV equation \cite{MNS} and the Dirac equation \cite{BCT}.

Within this framework, star graphs provide a prototypical model. Particularly, the issue of the existence of NLS ground states on star graphs has been widely investigated for the last ten years. As a first step, nonexistence on star graphs made of at least three half-lines with free (or Kirchhoff's) conditions at the origin was established in \cite{acfn12}, so that non-trivial conditions are required in order to have existence. The effect of an attractive linear delta interaction at the vertex, that corresponds to $q=2$ in \eqref{Fpq1}, was then studied in \cite{acfn_jde,acfn_aihp}, finding that ground states exist for small values of the mass only, and they bifurcate from the bound state of the corresponding linear Schr\"odinger Equation. Such ground states are always radial, in the sense that their restriction to every half-line of the graph always coincide with the same function. Moreover, the stability of the family of radial stationary states was proved even for the values of the mass for which such functions are no longer ground states. Further stability analysis on star graphs with linear pointwise interaction at the origin have then been developed in \cite{angulo1,angulo2,goloshchapova}. In \cite{cfn16} the result of existence of a nonlinear ground state bifurcating from the linear one was extended to the presence of a further linear potential.

Here we introduce nonlinear vertex conditions, more specifically conditions mimicking the nonlinear delta potential introduced in \cite{at_jfa}, and recently studied for the same problem on the line \cite{BD}.

The concentrated nonlinearity is nowadays a widely accepted model of the net effect of the confinement of charges in small regions \cite{jona,malomed}, as well as in the study of resonant tunneling \cite{nier}. Related models have been originally discussed in dimension one \cite{AT,at_jfa} and three \cite{ADFT1,ADFT2}, and more recently the analysis has been broadened to the two--dimensional case \cite{ACCT,ACCT2,CCT} and to non--compact metric graphs \cite{DT,ST-JDE,ST-NA,T-JMAA}. On the other hand, starting with the seminal paper \cite{tao}, the study of the interplay between different nonlinearities for NLS equations has been recently carried out for instance in \cite{cheng,jeanvisc,killip,lecoz,miao1,miao2,soave1,soave2}, focusing on the case of two different power nonlinearities. Particularly, \cite{jeanvisc,soave1,soave2} are devoted to the problem of the existence, the shape and the stability of prescribed mass ground states, defined as minimizers of the energy among all stationary states at given mass (see also the series of works \cite{bartsch1,bartsch2,bartsch3,bartsch4} and \cite{NTV} for the case of NLS systems).

In the present paper, we describe the effect of the combined action of two focusing nonlinearities again of the power type, with a substantial difference with respect to the cited ones, namely the fact that the second nonlinearity is concentrated at a point: the vertex of the graph. 

The interaction between this two kinds of nonlinearity has already been explored in the case of the line, keeping the two powers at the subcritical or $L^2$-critical level \cite{BD}. A non trivial interplay emerges in the critical case, with a modification of the critical mass. For the graph model treated here, the emerging scenario is considerably richer and surprising, as the first of our main theorems starts unravelling.

\begin{theorem}
	\label{exoutdiag}
	Let $p\in(2,6),\,q\in(2,4)$ and $q\neq\frac{p}{2}+1$. Then there exists a critical mass $\mu_{p,q}>0$ such that
	\begin{itemize}
		\item[(i)] if $q<\frac{p}{2}+1$, then ground states of \eqref{Fpq1} at mass $\mu$ exist if and only if $\mu\le\mu_{p,q}$;
		\item[(ii)] if $q>\frac{p}{2}+1$, then ground states of \eqref{Fpq1} at mass $\mu$ exist if and only if $\mu\ge\mu_{p,q}$.
	\end{itemize}	
	Furthermore, whenever they exist, ground states at prescribed mass are unique and they are radial and decreasing on $S_N$, in the sense that their restriction to each half--line of the graph corresponds to the same decreasing function on $\R^+$.
\end{theorem}
This result highlights the emergence of a natural comparison between the strength of the two nonlinearities. If $q < p/2 + 1$, then the point interaction is weak, so that it does not change the qualitative information gained in \cite{acfn_aihp}, that ground states exist only for masses below a critical value. Conversely, if $q > p/2 +1$, then the point interaction is strong enough to reverse the result, so that ground states exist only for masses above a critical value.

The threshold phenomena in Theorem \ref{exoutdiag} reveals a natural scaling for the doubly nonlinear problem on graphs. Heuristically, one may interpret such a feature as the result of the competition between the two terms defining $F_{p,q}$: the standard NLS energy on the one side
\begin{equation}
\label{energy_std}
\f12\int_\G|u'|^2\,dx-\f1p\int_\G|u|^p\,dx\,,
\end{equation}
and the delta nonlinearity on the other side
\[
-\f1q|u(0)|^q\,.
\]
It is well--known (see for instance \cite{ASTcpde}) that the unique minimizer of the standard NLS energy at mass $\mu$ on the real line, the so--called soliton $\phi$ as in \eqref{phi mu} below, verifies
\[
\f12\int_\R|\phi'|^2\,dx-\f1p\int_\R|\phi|^p\,dx\sim \mu^{2\beta+1}, \qquad\beta=\f{p-2}{6-p}\,,
\]
whereas
\[
\|\phi\|_{L^\infty(\R)}\sim \mu^\alpha,\qquad\alpha=\f2{6-p}\,.
\]
An elementary calculation then shows that
\[
\alpha q=2\beta+1\quad\Longleftrightarrow\quad q=\f p2+1\,,
\]
so that the comparison between the standard and the pointwise nonlinearity reported in Theorem \ref{exoutdiag} corresponds to the balance between the standard NLS energy and the maximal delta nonlinearity for the soliton on the line. However, at a first sight, the detailed result in Theorem \ref{exoutdiag}, with that unprecedented dependence of the existence of ground states on the mass, sounds puzzling and requires a qualitative explanation. As well understood since \cite{acfn12}, fixed a value for the mass, the existence of ground states for the NLSE on metric graphs is ruled by the competition between the soliton on the line and the standing wave with the lowest energy among all stationary states. Indeed, even though the presence of the vertex together with the prescription of continuity exclude the solitons from the family of possible competitors, one can always approximate its standard energy \eqref{energy_std} arbitrarily well, for instance by a sequence of solitons supported on one half--line and truncated near the vertex. The approximation is made better and better by moving faraway from the vertex. This reasoning shows that, if the lowest-energy standing wave wins the competition, namely, if its energy is lower than the standard energy of the soliton on the line \eqref{energy_std}, then the ground state exists. Viceversa, if the energy of the soliton is lower than the energy of every stationary state, then the ground state does not exist. This is nowadays well understood. 

Of course, the outcome of the competition between solitons and standing waves can depend on the mass, and this is the new phenomenon put in evidence, in a surprising way, by Theorem 1.1. We recall that such a phenomenon is not present in the analogous model on the line, since in that case the soliton is an admissible competitor and its doubly nonlinear energy is strictly smaller than the standard one \eqref{energy_std}. As a consequence, for the problem on the line a ground state exists for every value of the mass and is centred at the origin \cite{BD}. On the contrary, in Theorem 1.1 a dramatic dependence on the mass emerges as the graph structure plays a crucial role. Indeed, in order to lower the energy, from one side it is convenient to exploit the presence of the interaction at the vertex, but from the other it is convenient to escape it, as the proximity to the vertex increases the kinetic energy (quantitative estimates of this effect can be found by using rearrangement theory, see  \cite[Proposition 3.1]{ASTcpde}). By this observation, it becomes possible to interpret the role of the mass in Theorem 1.1: if $q > p/2 + 1$, then for large masses, and consequently large values at the vertex of the radial standing waves, the point interaction prevails and a ground state exists. On the other hand, for small mass, the high power on a small value at the vertex makes the contribution of the point interaction tiny, so that the standard nonlinearity prevails, escaping the vertex becomes convenient, and then there is no ground state. Viceversa, in the case of weaker delta interactions, namely $q < p/2 + 1$, for small mass, and then small value at the vertex, the lower power in the pointwise term gives to the interaction a larger size than that of the standard nonlinearity, so it is convenient to stay on the vertex and a ground state exists. On the other hand, if the mass is large, then the effect of the standard nonlinearity is magnified more, so that it is convenient to escape the vertex and there is no ground state. Of course, this description is qualitative, but is made rigorous by the computations in Section \ref{sec:main}.

It remains to investigate what happens at the threshold $q = p/2 +1$, and the answer is given by the following theorem, bearing another unexpected feature: ground states exist only for a small number of half-lines, and, notably, such a number does not depend on the mass.

\begin{theorem}
	\label{exdiag}
	Let $p\in(2,6)$ and $q=\frac{p}{2}+1$. Then there exists a critical number of half--lines $N_p\ge 2$ such that 
	\begin{itemize}
		\item[(i)]if $N\le N_p$, then  ground states of \eqref{Fpq1} at mass $\mu$ exist for every $\mu$;
		\item[(ii)]if $N>N_p$, then ground state of \eqref{Fpq1} at mass $\mu$ never exist.
	\end{itemize}
		Furthermore, whenever they exist, ground states at prescribed mass are unique and they are radial and decreasing on $S_N$, in the sense that their restriction to each half--line of the graph corresponds to the same decreasing function on $\R^+$.
\end{theorem}
The number $N_p$ of half-lines above which there is never a ground state is not presently known, but on the basis of numerical simulations (see Section \ref{sec:main}) we have evidence that it is at least three for every value of $p$. However, at the time being we can prove analytically only the following result.

\begin{proposition}
	\label{prop_Np}
	For every $p\in(2,6)$, let $N_p$ be as in Theorem \ref{exdiag}. Then:
	\begin{itemize}
		\item[(i)] $\lim_{p\to2^+}N_p=+\infty$;
		\item[(ii)] if $p\geq4$, then $N_p\in\{2,3,4\}$. Furthermore, there exists $\delta>0$ so that $N_p=3$ for every $p\in(4-\delta,4+\delta)\cup(6-\delta,6)$.
	\end{itemize}
\end{proposition}

To conclude, recall that to every critical point $u\in\HmuS$ of \eqref{Fpq1} subject to the mass constraint \eqref{mass} it corresponds a standing wave solution $\psi(x,t):=e^{i\omega t}u(x)$ to the time--dependent NLS equation on $S_N$
\begin{equation}
\label{time_nlse}
i\partial_t\psi(t,x)=-\partial_{xx}^2\psi(t,x)-|\psi(t,x)|^{p-2}\psi(t,x)-|\psi(t,x)|^{q-2}\delta_0\psi(t,x)\,.
\end{equation}
The standard stability theory in \cite{GSS} guarantees that $\psi$ is orbitally stable if and only if $u$ is a local minimizer of $F_{p,q}$ in $\HmuS$, so that Theorems \ref{exoutdiag}--\ref{exdiag} imply the existence of an orbitally stable standing wave of \eqref{time_nlse} whenever ground states exist.  Adapting the argument originally developed in \cite{acfn_jde}, the last proposition of this paper improves this result, by proving the orbital stability for all elements of the branch of radial stationary states, parametrized by the mass. Theorems \ref{exoutdiag}--\ref{exdiag} establish that for some values of the mass there is no ground state. However, even for those values the branch of radial standing waves exists, and we prove that they are always orbitally stable.

\begin{proposition}
	\label{prop_stab}
	For every $p\in(2,6)$, $q\in(2,4)$ and every $\mu>0$, the unique radially--symmetric stationary state of \eqref{time_nlse} is a local minimizer of $F_{p,q}$ in $\HmuS$, and the associated standing wave is always orbitally stable.
\end{proposition}

\medskip
The paper is organized as follows. Section \ref{sec:prel} reviews some preliminary results on nonlinear Schr\"odinger equations with the standard nonlinearity only, and it develops the analysis of the stationary states of the doubly nonlinear model on $S_N$. Within Section \ref{sec:ex_char} we derive a existence criterion and we show that, whenever they exist, ground states coincide with the unique radial and decreasing stationary state. Finally, Section \ref{sec:main} completes the proofs of Theorems \ref{exoutdiag}--\ref{exdiag} and of Proposition \ref{prop_Np}, while Section \ref{sec:stab} addresses that of Proposition \ref{prop_stab}.

\smallskip
\textbf{Notation.} In what follows, when denoting a norm, we omit the domain of integration whenever it is understood, writing for instance $\uLp$ instead of $\|u\|_{L^p(S_N)}$. The complete notation will be used if needed to avoid ambiguity.

\section{Preliminaries}
\label{sec:prel}

 In this section we briefly recall some well--known facts about nonlinear Schr\"odinger equations on the real line and we discuss some preliminary properties of stationary solutions on star graphs that will be important in the forthcoming analysis.

\subsection{NLSE with standard nonlinearity on the real line}
The minimization problem on the real line
\begin{equation*}
\label{infE}
\Eps(\mu):=\inf_{v\in H^1_\mu(\R)} E(v,\R)\,,
\end{equation*}
where $E:H^1(\R)\to\R$ is the NLS energy functional involving the standard nonlinearity only
\begin{equation}
\label{E}
E(v):=\frac{1}{2}\int_\R\abs{u'(x)}^2\,dx-\frac{1}{p}\int_\R\abs{u(x)}^p\,dx,
\end{equation} is nowadays classical (see for instance \cite{cazenave}).

Standard variational arguments show that ground states are solutions to the stationary nonlinear Schr\"odinger equation
\begin{equation}
\label{nlse}
u''+|u|^{p-2}u=\omega u\qquad\text{on }\R
\end{equation}
for some $\omega>0$. In fact, for every $\omega>0$ the unique (up to translations) positive solution $\phi_\omega\in H^1(\R)$ of \eqref{nlse} is
\begin{equation}
\label{soliton}
\phi_\omega(x)=\left[\frac{p}{2}\omega\left(1-\tanh^2\left(\left(\frac{p}{2}-1\right)\sqrt{\omega}|x|\right)\right)\right]^{\frac{1}{p-2}}\,.
\end{equation}
The mass of $\phi_\omega$ is given explicitly by
\begin{equation}
\label{mass phi w}
\|\phi_\omega\|_2^2=\frac{4\left(\frac{p}{2}\right)^\frac{2}{p-2}\omega^{\frac{6-p}{2(p-2)}}}{p-2}\int_{0}^1(1-s^2)^{\frac{4-p}{p-2}}\,ds,
\end{equation}
which is a continuous, strictly increasing and unbounded function of $\omega$. Therefore, for every $\mu>0$ there exists a unique $\omega(\mu)$ such that $\phi_{\omega(\mu)}$ is the unique (up to translations) positive ground state of $E$ in $H_\mu^1(\R)$. Such ground states are called solitons, and their dependence on $\mu$ is given by 
\begin{equation}
	\label{phi mu}
	\phi_{\omega(\mu)}(x)=C_p\mu^\alpha\sech^{\f\alpha\beta}\left(c_p\mu^\beta x\right)\,,\qquad\alpha=\f2{6-p},\,\beta=\f{p-2}{6-p}\,,
\end{equation}
where $C_p,\,c_p>0$ depends on $p$ only, and one can easily compute
\begin{equation}
	\label{E phi mu}
	\Eps(\mu)=E(\phi_{\omega(\mu)})=-\theta_p\mu^{2\beta+1}
\end{equation}
where $\theta_p>0$ depends on $p$ only.

In the following, for $u\in H^1(S_N)$, we will denote by $E(u,S_{N})$ the analogous standard NLS energy as in \eqref{E}
\[
E(u,S_N):=\f12\int_{S_N}|u'|^2\,dx-\f1p\int_{S_N}|u|^p\,dx=\f12\sum_{i=1}^N\int_{\mathcal{H}_i}|u_i'|^2\,dx-\f1p\sum_{i=1}^N \int_{\mathcal{H}_i}|u_i|^p\,dx\,.
\]
Note that, for every positive $u\in\HmuS$, it holds
\begin{equation}
\label{Esn geq Er}
	E(u,S_N)>\Eps(\mu)\,.
\end{equation}
Indeed, if $u$ is compactly supported on a unique half--line of $S_N$, then one can regard it as a compactly supported function in $H_\mu^1(\R)$ and \eqref{Esn geq Er} is immediate since only solitons attain $\Eps(\mu)$. On the contrary, if $u\not\equiv0$ on at least two half--lines of $S_N$, then there exist infinitely many values $t$ in the image of $u$ such that the number of pre--images
\[
N(t):=\#\{x\in S_N\,:\,u(x)=t\}>2\,.
\]
If $u(0)\neq0$ then this is trivial, as $u$ tends to $0$ along each half--line. Similarly, if $u(0)=0$, then all values realized by $u$ in a suitably small neighbourhood of the origin are attained at least twice the number of half--lines that belong to the support of $u$. Hence, letting $\widehat{u}\in H_\mu^1(\R)$ denote the symmetric rearrangement of $u$, it follows (see \cite[Proposition 3.1]{ASTcpde})
\[
\|u'\|_{L^2(S_N)}>\|\widehat{u}'\|_{L^2(\R)}\qquad\text{and}\qquad\|u\|_{L^p(S_N)}=\|\widehat{u}\|_{L^p(\R)}\,,\quad p\geq 1\,,
\]
yielding again \eqref{Esn geq Er}.
 
\subsection{Some remarks on stationary states}

Computing the Euler--Lagrange equations associated to the energy \eqref{Fpq1} and to the mass constraint \eqref{mass}, it turns out that ground states of \eqref{Fpq1} at given mass are solutions to the system
\begin{equation}
\label{EulLag}
\begin{cases}
u_i''+u_i\abs{u_i}^{p-2}=\omega u_i\quad\text{on $\HH_i$ for all $i=1,\dots,N$},\\
\sum_{i=1}^N \frac{du_i}{dx_i}(0^+)=-u(0)\abs{u(0)}^{q-2},
\end{cases}
\end{equation}
for some Lagrange multiplier $\omega>0$. 

The next proposition provides a complete characterization of the set of positive solutions of \eqref{EulLag} in $H^1(S_N)$.

\begin{figure}[t]
	\centering
	\subfloat[][]{
	\includegraphics[width=0.4\textwidth]{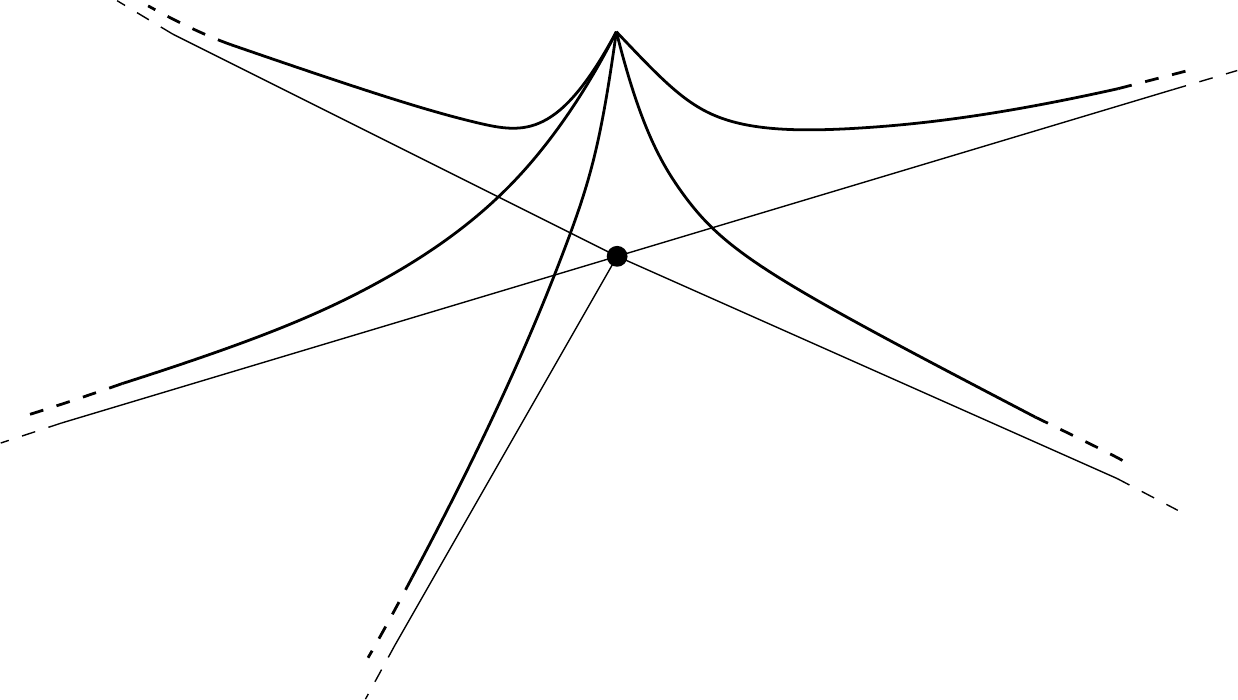}}\qquad
	\subfloat[][]{
	\includegraphics[width=0.4\textwidth]{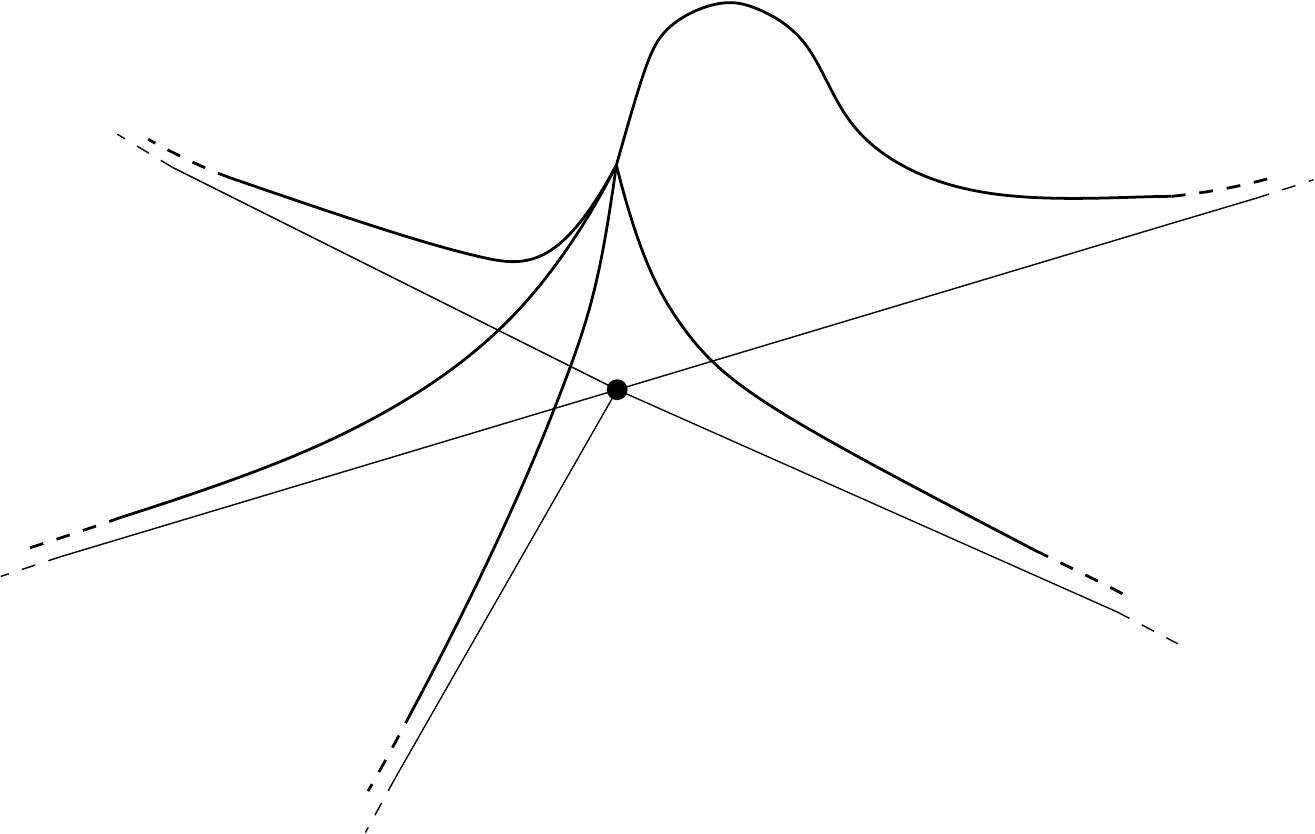}}\qquad
	\subfloat[][]{
	\includegraphics[width=0.45\textwidth]{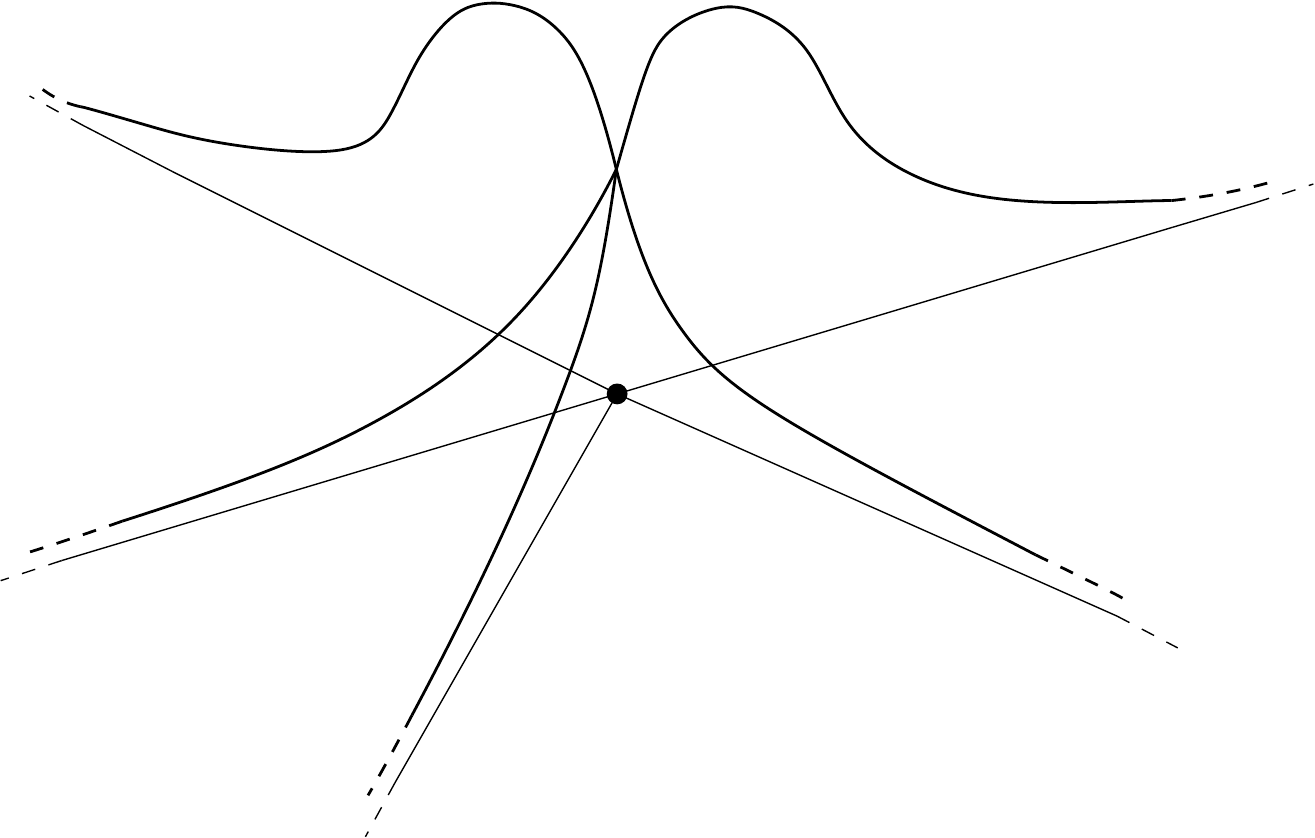}}
	\caption{The stationary states of Proposition \ref{prop_stat} on $S_5$: $\eta_0^\omega$ (A), $\eta_1^\omega$ (B) and $\eta_2^\omega$ (C).}
	\label{fig:stat}
\end{figure}

\begin{proposition}
	\label{prop_stat}
	Let $\omega>0$. Then the set $\mathcal{S}_{\omega}$ of positive solutions of \eqref{EulLag} in $H^1(S_N)$ is given by
	\[
	\mathcal{S}_{\omega}=\Bigg\{\eta_{J}^\omega\,:\,J\in\N,\,0\leq J\leq\f{N-1}2\Bigg\}\,.
	\]
	Here, $\eta_J^\omega\in H^1(S_N)$ is such that
	\begin{equation}
	\label{etaJ}
	\eta_J^\omega(\cdot)=\begin{cases}
	\phi_\omega(\cdot-a) & \text{on }J\text{ half--lines of }S_N\\
	\phi_\omega(\cdot+a) & \text{on the other }N-J\text{ half--lines of }S_N\,,
	\end{cases}
	\end{equation}
	where $a>0$ is the unique positive solution to
		\begin{equation}
	\label{tanhJ2}
	\frac{\tanh\left(\left(\frac{p}{2}-1\right)\sqrt{\omega}a\right)}{\left(1- \tanh^2\left(\left(\frac{p}{2}-1\right)\sqrt{\omega}a\right)\right)^{\frac{q-2}{p-2}}}=\frac{\left(\frac{p}{2}\right)^{\frac{q-2}{p-2}}\omega^{\frac{2q-2-p}{2(p-2)}}}{N-2J}\,.
	\end{equation}
	Furthermore, for every $J$, the function $\omega\mapsto\|\eta_J^\omega\|_2^2$ is continuous and strictly increasing.
\end{proposition}

\begin{proof}
	Relying on the discussion of the previous subsection, it is immediate to see that, on each half--line $\HH_i$, $i=1,\,\dots,\,N$, solutions of \eqref{EulLag} coincide with the restriction $\phi_\omega(x-a_i)$ of the soliton $\phi_\omega$, for suitable $a_i\in\R$. 
	Since the continuity condition at the origin is prescribed, then $\phi_\omega(-a_i)=\phi_\omega(-a_j)$ for $1\le i,j\le N$, therefore $a_i=\epsilon_i a$, where $a>0$ and $\epsilon_i=\sgn(a_i)$ for all $1\le i\le N$.
	
	Moreover, by \eqref{soliton}, the second line of \eqref{EulLag} becomes
	\begin{equation}
	\label{tanh1}
	\sqrt{\omega}\tanh\left(\left(\frac{p}{2}-1\right)\sqrt{\omega}a\right)\sum_{i=1}^N  \epsilon_i=-\phi_\omega(a)^{q-2}<0,
	\end{equation}
	thus implying $\sum_{i=1}^N\epsilon_i<0$, namely the number of positive $\epsilon_i$ cannot exceed $\f{N-1}2$.
	
	Let then $J$ be a given integer such that $0\leq J\leq\f{N-1}2$ and let $\epsilon_i=1$ if and only if $i\leq J$. Since then $\sum_{i=1}^N=2J-N$, \eqref{tanh1} reads
	\begin{equation}
	\label{tanhJ1}
	\tanh\left(\left(\frac{p}{2}-1\right)\sqrt{\omega}a\right)=\frac{\phi_\omega(a)^{q-2}}{\sqrt{\omega}(N-2J)}\,,
	\end{equation}
	so that \eqref{tanhJ2} is proved.

	Setting $t:=\tanh\left(\left(\frac{p}{2}-1\right)\sqrt{\omega}a\right)$, then $t\in[0,1)$ and we can rewrite \eqref{tanhJ2} as 
	\begin{equation}
	\label{ftJ}
	f(t):=\frac{t}{(1-t^2)^{\frac{q-2}{p-2}}}=\frac{\left(\frac{p}{2}\right)^{\frac{q-2}{p-2}}\omega^{\frac{2q-2-p}{2(p-2)}}}{N-2J}\,.
	\end{equation}
	As $f(0)=0$, $\lim_{t\to1^-}f(t)=+\infty$ and $f'(t)>0$ for every $t\in(0,1)$, it follows that there exists a unique $t\in(0,1)$ for which \eqref{ftJ} is satisfied, in turn showing that for every $\omega>0$ and every integer $0\leq J\leq\f{N-1}2$, there exists a unique positive solution $\eta_J^\omega$ of \eqref{EulLag}, up to exchange of edges.
	
	Relying again on \eqref{soliton}, we get
	\begin{equation}
	\label{mass1}
	\|\eta_J^\omega\|_2^2=\frac{2(\frac{p}{2})^\frac{2}{p-2}\omega^{\frac{6-p}{2(p-2)}}}{p-2}\left[2J\int_{0}^1(1-s^2)^{\frac{4-p}{p-2}}\,ds+(N-2J)\int_{t}^1(1-s^2)^{\frac{4-p}{p-2}}\,ds\right].
	\end{equation}
	Differentiating with respect to $\omega$ yields
	\begin{equation}
	\label{dmu dw}
	\begin{split}
	\f d{d\omega}\|\eta_J^\omega\|_2^2=&\f{\left(\f p2\right)^{\f2{p-2}}(6-p)}{(p-2)^2}2J\omega^{\f{6-p}{2(p-2)}-1}\int_{0}^1(1-s^2)^{\frac{4-p}{p-2}}\,ds\\
	+&\f{2\left(\f p2\right)^{\f2{p-2}}}{p-2}(N-2J)\omega^{\f{6-p}{2(p-2)}-1}\left[\f{6-p}{2(p-2)}\int_{t}^1(1-s^2)^{\frac{4-p}{p-2}}\,ds-\omega(1-t^2)^{\f{4-p}{p-2}}t'(\omega)\right]\,,
	\end{split}
	\end{equation}
	where, by \eqref{ftJ},
	\begin{equation}
	\label{t'}
	t'(\omega)=\f{\left(\f p2\right)^{\f{q-2}{p-2}}(2q-2-p)}{2(p-2)(N-2J)}\omega^{\f{2q-2-p}{2(p-2)}-1}\f{(1-t^2)^{\f{q-2}{p-2}+1}}{t^2\left(2\f{q-2}{p-2}-1\right)+1}\,.
	\end{equation}
	
	Note that the first term in the sum on the right hand side of \eqref{dmu dw} is strictly positive for every $\omega>0$. Furthermore, plugging \eqref{t'} into \eqref{dmu dw} and making use of \eqref{ftJ} allows to rewrite the term between square brackets as
	\[
	\f{6-p}{2(p-2)}\int_{t}^1(1-s^2)^{\frac{4-p}{p-2}}\,ds-\f{2q-2-p}{2(p-2)}\f{t(1-t^2)^{\f{2}{p-2}}}{t^2\left(2\f{q-2}{p-2}-1\right)+1}\,,
	\]
	which can be proved to be strictly positive for every $t\in(0,1)$ by the same calculations in \cite[pp.11--12]{BD}. Hence, $\f d{d\omega}\|\eta_J^\omega\|_2^2>0$ for every $\omega>0$ and $\|\eta_J^\omega\|_2^2$ is a strictly increasing function of $\omega$.
\end{proof}

The next straightforward corollary describes the set of stationary states at prescribed mass.
\begin{corollary}
	\label{cor_stats}
	Let $\mu>0$ be fixed. Then for every integer $J$ such that $0\leq J\leq \f{N-1}2$ there exists a unique $\omega>0$ such that $\eta_J^\omega\in\HmuS$. Moreover, it holds
	\begin{equation}
	\label{Fpqstationary}
	F_{p,q}(\eta_{J}^\omega)=-\frac{6-p}{2(p+2)}\omega\mu+\left(\frac{2}{p+2}-\frac{1}{q}\right)\abs{\eta_{J}^\omega(0)}^q\,.
	\end{equation}
\end{corollary}
\begin{proof}
	Fix $J$ integer such that $0\leq J\leq \f{N-1}{2}$. By Proposition \ref{prop_stat}, $\|\eta_J^\omega\|_2^2$ is an unbounded, strictly increasing continuous function of $\omega$, thus implying that for every $\mu>0$ there exists a unique value of $\omega$ for which $\|\eta_J^\omega\|_2^2=\mu$, i.e. $\eta_J^\omega\in\HmuS$.
	
	Since $\eta_J^\omega$ solves \eqref{EulLag}, multiplying the equation by $(\eta_J^\omega)'$ and, for every $x\in S_N$, integrating on $[x,+\infty)$ one obtains
	\[
	\f12\left((\eta_J^{\omega})'\right)^2(x)=\f\omega2\left(\eta_J^\omega\right)^2(x)-\f1p\left(\eta_J^\omega\right)^p(x)\,,
	\]
	so that integrating on $S_N$ gives
\begin{equation}
	\label{energycons}
	\frac{1}{2}\|\left(\eta_J^\omega\right)'\|_2^2=\frac{\omega}{2}\|\eta_J^\omega\|_2^2-\frac{1}{p}\|\eta_J^\omega\|_p^p\,.
\end{equation}
	Furthermore, multiplying the first line of \eqref{EulLag} by $\eta_J^\omega$, integrating over $S_N$ and making use of the second line of \eqref{EulLag} yields
	\begin{equation}
	\label{nehariman}
	\abs{\eta_J^\omega(0)}^q-\|\left(\eta_J^\omega\right)'\|_2^2+\|\eta_J^\omega\|_p^p-\omega\|\eta_J^\omega\|_2^2=0.
	\end{equation}
	Combining \eqref{energycons} and \eqref{nehariman} leads to \eqref{Fpqstationary} and we conclude.
\end{proof}

\section{Existence criterion and characterization of ground states}
\label{sec:ex_char}
In this section we provide a sufficient condition granting existence of ground states of $F_{p,q}$ in $\HmuS$ and prove that, whenever they exist, such ground states must be monotonically decreasing radial functions on $S_N$.

Let us begin with a compactness result. To this aim we recall the Gagliardo--Nirenberg inequalities
\begin{equation}
\label{GNp}
\uLp^p\le K_p \|u\|_2^{\frac{p}{2}+1}\|u'\|_2^{\frac{p}{2}-1}\,,\qquad p\geq2\,,
\end{equation}
where $K_p>0$ depends only on $p$, and   
\begin{equation}
\label{GNinf}
\|u\|_\infty^2\le \|u\|_2\|u'\|_2\,,
\end{equation}
holding for every $u\in H^1(S_N)$ (we refer to \cite{AST} for a proof of these inequalities on general metric graphs).

\begin{proposition}
\label{compactth}
For every $\mu>0$ it holds
\begin{equation}
\label{F leq E}
\F_{p,q}(\mu)\leq\Eps(\mu)\,.
\end{equation}
Furthermore, if $\F_{p,q}(\mu)<\Eps(\mu)$, then ground states of \eqref{Fpq1} at mass $\mu$ exist.
\end{proposition}
\begin{proof}
Let us first prove \eqref{F leq E}. For every $\varepsilon>0$, let $v_\varepsilon:=\kappa_\varepsilon(\phi_\mu-\varepsilon)_+$, where $\phi_\mu$ is as in \eqref{phi mu} and $\kappa_\varepsilon>0$ is chosen to guarantee $\|v_\varepsilon\|_{L^2(\R)}^2=\mu$, so that $v_\varepsilon\in H_\mu^1(\R)$. Since $\|v_\varepsilon\|_{L^q(\R)}\to\|\phi_\omega\|_{L^q(\R)}$ as $\varepsilon\to0$, for every $q\geq1$, then $\kappa_\varepsilon\to1$ for $\varepsilon\to0$, and we get 
\[
\mathcal{E}(\mu)\leq E(v_\varepsilon,\R)=\f12\kappa_\varepsilon^2\int_R|v_\varepsilon'|^2\,dx-\f1p\kappa_\varepsilon^{ p}\int_\R|v_\varepsilon|^p\,dx\leq E(\phi_\omega,\R)+o(1)=\mathcal{E}(\mu)+o(1)
\]
for $\varepsilon$ small enough, making use also of $\|v_\varepsilon'\|_{L^2(\R)}\leq\|\phi_{\omega}\|_{L^2(\R)}$. Hence, $E(v_\varepsilon,\R)\to\mathcal{E}(\mu)$ as $\varepsilon\to0$. Moreover, $v_\varepsilon$ has compact support, so that one can think of it as supported on any given half--line of $S_N$. We thus have
\[
\Eps(\mu)=\lim_{\varepsilon\to0^+}E(v_\varepsilon,\R)=\lim_{\varepsilon\to0^+}F_{p,q}(v_\varepsilon)\geq\F_{p,q}(\mu)\,,
\]
so \eqref{F leq E} is proved.

Assume now that $\F_{p,q}(\mu)<\Eps(\mu)$ and let $(u_n)\subset H^1_\mu(S_{N})$ be a
minimizing sequence for $F_{p,q}$. Plugging \eqref{GNp} and \eqref{GNinf} into the definition of $F_{p,q}$ gives
\begin{equation*}
F_{p,q}(u_n)\ge\frac{1}{2}\|u_n'\|_2^2-\frac{K_p}{p}\mu^\frac{p+2}{4}\|u_n'\|_2^{\frac{p}{2} -1}-\frac{1}{q}\mu^\frac{q}{4}\|u_n'\|_2^\frac{q}{2}
\end{equation*}
which, since $p\in(2,6)$, $q\in(2,4)$, ensures that $(u_n)$ is bounded in $H^1(S_N)$.
Therefore there exists $u\in H^1(S_N)$ such that, up to subsequences, $u_n \deb u$ weakly in $H^1(S_{N})$, $u_n\rightarrow u$ in $L^\infty_{loc}(S_{N})$ and consequently $u_n\rightarrow u$ a.e. in $S_{N}$.

Set $m:=\|u\|_2^2$. By weak lower semicontinuity, we have $m\le \mu$.

Assume $m=0$, that is $u\equiv 0$. Then $u_n(0)\to0$ as $n\to+\infty$, so that recalling \eqref{Esn geq Er} leads to
\begin{equation*}
\Eps(\mu)>\F_{p,q}(\mu)=\lim_n F_{p,q}(u_n)=\lim_n E(u_n,S_{N})\ge\Eps(\mu),
\end{equation*}
i.e., a contradiction. Hence, $u\not\equiv0$ on $S_N$.

Suppose then that $0<m<\mu$. By weak convergence in $H^1(S_N)$ of $u_n$ to $u$, we get $\|u_n-u\|_2^2=\mu-m+o(1)$ for $n\to+\infty$. On the one hand, since $p,q>2$ and $\frac{\mu}{\|u_n-u\|_2^2}>1$ for $n$ sufficiently large,
\begin{equation*}
\begin{split}
&\F_{p,q}(\mu)\le F_{p,q}\left(\sqrt{\frac{\mu}{\|u_n-u\|_2^2}}(u_n-u)\right)\\
&=\f12\f{\mu}{\|u_n-u\|_2^2}\|u'_n-u'_n\|_2^2-\frac{1}{p}\left(\frac{\mu}{\|u_n-u\|_2^2}\right)^{\frac{p}{2}}\|u_n-u\|_p^p\\
&-\frac{1}{q}\left(\frac{\mu}{\|u_n-u\|_2^2}\right)^{\frac{q}{2}}\abs{u_n(0)-u(0)}^q<\frac{\mu}{\|u_n-u\|_2^2}F_{p,q}(u_n-u), 
\end{split}
\end{equation*}
so that
\begin{equation}
\label{Fun-est}
\liminf_n F_{p,q}(u_n-u)\ge \frac{\mu-m}{\mu}\F_{p,q}(\mu).
\end{equation}
On the other hand, an analogous reasoning leads to
\begin{equation*}
\F_{p,q}(\mu)\le F_{p,q}\left(\sqrt{\frac{\mu}{\|u\|_2^2}}\,u\right)<\frac{\mu}{\|u\|_2^2}F_{p,q}(u),
\end{equation*}
so
\begin{equation}
\label{Fu-est}
F_{p,q}(u)>\frac{m}{\mu}\F_{p,q}(\mu).
\end{equation}
Moreover, it holds
\begin{equation}
\label{FpqBrezLieb}
F_{p,q}(u_n)=F_{p,q}(u_n-u)+F_{p,q}(u)+o(1).
\end{equation}
Indeed, by $u'_n\deb u'$ weakly in $L^2(S_N)$ and $u_n\to u$ in $L_{\text{loc}}^\infty(S_N)$, we have
$\|u'_n-u'\|_2^2=\|u'_n\|_2^2-\|u'\|_2^2+o(1)$ and $|(u_n-u)(0)|^q=o(1)$ as $n$ is large enough. Furthermore, owing to the Brezis-Lieb lemma \cite{BL},
\[
\|u_n\|_p^p=\|u_n-u\|_p^p+\|u\|_p^p+o(1).
\]
Using now \eqref{Fun-est}, \eqref{Fu-est} and \eqref{FpqBrezLieb}, we get
\begin{equation*}
\begin{split}
\F_{p,q}(\mu)&=\lim_n F_{p,q}(u_n)=\lim_n F_{p,q}(u_n-u)+F_{p,q}(u)\\
&>\frac{\mu-m}{\mu}\F_{p,q}(\mu)+\frac{m}{\mu}\F_{p,q}(\mu)=\F_{p,q}(\mu),
\end{split}
\end{equation*}
which is again a contradiction.

Henceforth, $m=\mu$ and $u\in\HmuS$. In particular, $u_n\rightarrow u$ in $L^2(S_{N})$ so that, $(u_n)$ being bounded in $L^\infty(S_{N})$, $u_n\rightarrow u$ in $L^p(S_{N})$ as $n\to+\infty$. Thus, by weak lower semicontinuity
\[
F_{p,q}(u)\leq\lim_n F_{p,q}(u_n)=\F_{p,q}(\mu)\,,
\]
that is $u$ is a ground state of $F_{p,q}$ at mass $\mu$.
\end{proof}
\begin{corollary}
\label{compactcor}
Let $\mu>0$ be fixed. If there exists $u\in H^1_\mu(S_{N})$ such that $F_{p,q}(u)\le \Eps(\mu)$, then ground states of \eqref{Fpq1} at mass $\mu$ exist.
\end{corollary}
\begin{proof}
If $\F_{p,q}(\mu)=F_{p,q}(u)$ then $u$ is a ground state at mass $\mu$. Otherwise, $\F_{p,q}(\mu)<F_{p,q}(u)\leq\Eps(\mu)$ and a ground state of \eqref{Fpq1} at mass $\mu$ exists by Proposition \ref{compactth}.
\end{proof}
Once existence of ground states is granted, the following proposition ensures uniqueness and provides a complete characterization of their symmetry properties.
\begin{proposition}
\label{GSsym}
Let $\mu>0$ be such that $\F_{p,q}(\mu)$ is attained. Then the unique positive ground state of $F_{p,q}$ at mass $\mu$ is the stationary state $\eta_0^{\omega(\mu)}$ such that $\eta_0^{\omega(\mu)}\in\HmuS$. 
\end{proposition}
To prove this proposition we need two auxiliary lemmas.
\begin{lemma}
\label{rearrang}
Let $u\in H^1_\mu(S_N)$, $u\geq0$. Then there exists $u^*\in\HmuS$, $u^*\geq0$, such that $F_{p,q}(u^*)\le F_{p,q}(u)$ and either 
\begin{itemize}
	\item[(i)] $u^*$ is symmetric with respect to the vertex and monotonically decreasing on each half--line, or
	\item[(ii)] $u^*$ is symmetric with respect to the vertex and monotonically decreasing on $N-1$ half--lines of $S_N$, whereas on the remaining half--line it is non--decreasing from the origin to a unique maximum point and then non--increasing from this point to infinity.
\end{itemize}
 \end{lemma} 
\begin{proof}
Given a non-negative function $u\in H^1_\mu(S_N)$, suppose first that $\|u\|_\infty$ is attained at least once on every half--line of $S_N$ (this hypothesis includes the particular case in which the maximum is attained at the origin). Then all the values in the image of $u$ are attained at least $N$ times on the graph. Thus, letting $u^*\in\HmuS$ be the symmetric rearrangement of $u$ on $S_N$ as defined in \cite[Appendix A]{acfn_jde}, standard properties of rearrangements give
\[
\|u'\|_2\geq\|(u^*)'\|_2\,,\qquad\|u\|_p=\|u^*\|_p\,,\qquad|u^*(0)|=\|u\|_\infty\,,
\]
the inequality being strict unless $u$ is symmetric with respect to the vertex and monotonically decreasing on each half--lines. Hence, $F_{p,q}(u^*)\leq F_{p,q}(u)$ and $u^*$ is as in \emph{(i)}.

Assume now that $\|u\|_\infty$ is not attained on every half--line of the graph. Let us discuss separately the cases $u(0)=0$ and $u(0)>0$.

If $u(0)=0$, then let $\|u\|_\infty$ be attained at $x_{1}\in \HH_{1}$ and denote by $I$ the interval connecting the origin to $x_{1}$ along $\HH_1$. Consider the following construction. 

First, let $\widetilde{u}\in H^1(0,x_1)$ be the monotone rearrangement of the restriction $u_{|_I}$ of $u$ to $I$. By definition of monotone rearrangement and the P\'olya--Szeg\H{o} inequality, we have $\tilde{u}(0)=\|u\|_\infty$, $\tilde{u}(x_{1})=0$ and
\[
\|u'\|_{L^2(I)}\geq\|\widetilde{u}'\|_{L^2(0,x_1)}\,,\qquad\|u\|_{L^p(I)}=\|\widetilde{u}\|_{L^p(0,x_1)}\quad\quad\forall p\geq1\,.
\]
Secondly, let $\bar{u}\in H^1(\R^+)$ be the monotone rearrangement of the restriction $u_{|_{{S_{N}\setminus I}}}$ of $u$ to $S_N\setminus I$, so that $\bar{u}(0)=\|u\|_\infty$ and, by usual estimates that hold for monotone rearrangements,
\[
\|u'\|_{L^2(S_N\setminus I)}\geq\|\bar{u}'\|_{L^2(\R^+)}\,,\qquad\|u\|_{L^p(S_N\setminus I)}=\|\bar{u}\|_{L^p(\R^+)}\quad\quad\forall p\geq1\,.
\]
Define then $u^*:S_N\to\R$ to be
 \begin{equation*}
 u^{*}(x):=
 \begin{cases}
 \tilde{u}(x_{1}-x)\quad&x\in I,\\
 \bar{u}(x-x_{1}) &x\in \HH_{1}\setminus I,\\
 0 &\text{otherwise.}
 \end{cases}
 \end{equation*} 
By construction, it follows that $u^*\in\HmuS$, $u^*$ is as in \emph{(ii)} and $F_{p,q}(u^*)\leq F_{p,q}(u)$.
  
To conclude, it remains to deal with the case $u(0)>0$. Given this, let $J:=\{x\in S_N\,:\,u(x)>u(0)\}$ and let $u_{|_J}$ be the restriction of $u$ to $J$. Note that $u(J)=(u(0),\|u\|_\infty]$ is connected and every $t\in u(J)$ is attained at least twice on $S_N$, except possibly $\|u\|_\infty$. Hence, denoting by $\widehat{u}\in H^1(-L,L)$ the symmetric rearrangement of $u_{|_J}$ on the interval $(-L,L)$, with $L:=\f{|J|}2$, we have (see \cite[Proposition 3.1]{AST}) 
\[
\|u'\|_{L^2(J)}\geq\|\widehat{u}'\|_{L^2(-L,L)}\,,\qquad\|u\|_{L^p(J)}=\|\widehat{u}\|_{L^p(-L,L)}\quad\forall p\geq1,\qquad\widehat{u}(0)=\|u\|_\infty\,.
\]
Similarly, $u(S_N\setminus J)\subseteq[0,M]$ is connected and every value $t\in S_N\setminus J$ is attained at least $N$ times (once on each half--line). Therefore, letting $u^\dagger\in H^1(S_N)$ be the symmetric rearrangement on $S_N$ of $u_{\mid S_N\setminus I}$ as in \cite[Appendix A]{acfn_jde}, we get
\[
\|u'\|_{L^2(S_N\setminus J)}\geq\|(u^\dagger)'\|_{L^2(S_N)}\,,\qquad\|u\|_{L^p(S_N\setminus J)}=\|u^\dagger\|_{L^p(S_N)}\quad p\geq1,\qquad u^\dagger(0)=u(0)\,.
\]
Let then $I$ be the interval $[0,2L]$ along $\HH_{1}$, and set $u^*:S_N\to\R$ to be
\[
u^*(x):=\begin{cases}
\widehat{u}(x-L) & x\in I\\
u^\dagger(x-2L) & x\in\HH_{1}\setminus I\\
u^\dagger(x) & \text{otherwise}\,.
\end{cases}
\]
By construction, $u^{*}\in H^{1}_{\mu}(S_{N})$, $u^*$ is as in \emph{(ii)} and $F_{p,q}(u^*)\leq F_{p,q}(u)$, so that the proof is complete.
\end{proof}
\begin{lemma}
\label{propervar}
Consider $u\in H^1_\mu(S_N)$ that does not have a local maximum point at the origin. Then there exists $v\in H^1_\mu(S_N)$ such that $F_{p,q}(v)<F_{p,q}(u)$.
\end{lemma}
\begin{proof}
Given $u\in\HmuS$, let $\overline{u}$ be the restriction of $u$ to $\HH_1\cup\HH_2$ and $\widetilde{u}$ be the restriction of $u$ to $S_N\setminus(\HH_{1}\cup\HH_2)$. Set also
\begin{equation*}
\overline{\mu}:=\int_{\HH_1\cup\HH_2}\abs{\overline{u}}^2\,dx,\quad \widetilde{\mu}:=\int_{S_N\setminus(\HH_{1}\cup\HH_2)}\abs{\widetilde{u}}^2\,dx.
\end{equation*}
For $\varepsilon>0$ small enough, let $\nu\in(-\varepsilon,\varepsilon)$. Since $u$ has no maximum point at the origin, define
\begin{equation*}
u_\nu(x):=\begin{cases}
\sqrt{\frac{\overline{\mu}+\nu}{\overline{\mu}}}\overline{u}(x+T(\nu)) & x\in\HH_{1}\cup\HH_2\\
\sqrt{\frac{\widetilde{\mu}-\nu}{\widetilde{\mu}}}\widetilde{u}(x) & x\in S_N\setminus(\HH_1\cup\HH_2)\,,
\end{cases} 
\end{equation*}
where the shift $T(\nu)$ is such that $T(0)=0$ and $u_\nu$ is continuous at the origin. One has $\|u_\nu\|_2^2=\mu$ for every $\nu$ and
\begin{equation*}
\frac{d^2}{d\nu^2}F_{p,q}(u_\nu)\Big|_{\nu=0}=-\frac{p}{2}\left(\frac{p}{2}-1\right)\int_{S_N}\abs{u}^p\,dx-\frac{q}{2}\left(\frac{q}{2}-1\right)\abs{u(0)}^q<0,
\end{equation*}
so that, choosing $\varepsilon$ small enough and setting $v:=u_{\nu}$ for any $\nu\in(-\varepsilon,\varepsilon)$, we conclude.
\end{proof}
\begin{proof}[Proof of Proposition \ref{GSsym}]
Note that, by Lemma \ref{rearrang}--\ref{propervar}, any ground state $u\in\HmuS$ of $F_{p,q}$ at mass $\mu$ has to be symmetric with respect to the origin and non--increasing on each half--line. Since the unique solution to \eqref{EulLag} fulfilling these properties is $\eta_0^\omega$ and given that, by Corollary \ref{cor_stats}, there exists a unique $\omega>0$ in \eqref{EulLag} for which $\eta_0^\omega$ belongs to $\HmuS$, the proof is complete.
\end{proof}
Note that, whenever $\F_{p,q}(\mu)$ is attained, Proposition \ref{GSsym} entails
\[
\F_{p,q}(\mu)=\inf_{v\in H^1_{\mu, \text{rad}}(S_N)} F_{p,q}(v)
\]
where $H^{1}_{\mu,\text{rad}}=\big\{v\in H^{1}_{\mu}(S_{N})\,:\,v\,\, \text{is symmetric with respect to the origin of }S_N\big\}$. We conclude this section with the next proposition concerning radial ground states of $F_{p,q}$, establishing some properties that will be useful in what follows.
\begin{proposition}
	\label{prop_radial} For every $\mu>0$, the minimization problem
	\begin{equation}
	\label{radinf}
	\F_{p,q}^{\text{rad}}(\mu):=\inf_{v\in H^1_{\mu, \text{rad}}(S_N)} F_{p,q}(v)
	\end{equation}
	is always attained by the unique stationary state $\eta_0^{\omega(\mu)}$.
\end{proposition}
\begin{proof}
	First notice that
	\begin{equation}
		\label{neginf}
		\F_{p,q}^{\text{rad}}(\mu)<0
	\end{equation}
	for every $\mu>0$. Indeed, let us denote by $\varphi\in H^1(\R^+)$ the restriction to $\R^+$ of the soliton $\phi_{\omega\left(\f{2\mu}{N}\right)}$ of mass $\f{2\mu}{N}$, so that $\|\varphi\|_{L^2(\R^+)}^2=\f\mu N$. Setting $v:\equiv\varphi$ on each half--line of $S_N$, we get $v\in H_{\mu,\text{rad}}^1(S_N)$ and 
	\[
	\F_{p,q}^{\text{rad}}(\mu)\leq F_{p,q}(v)<NE(\varphi,\R^+)=-\left(\f2N\right)^{2\beta}\theta_p\mu^{2\beta+1}<0\,.
	\]
	Moreover, the minimization problem \eqref{radinf} is equivalent to minimize $F_{p,q}$ among all functions $u\in H_{\mu,\text{rad}}^1(S_N)$ non--increasing on each half--line. Indeed, arguing as in the first part of the proof of Lemma \ref{GSsym}, it is possible to construct $u^*\in H_{\mu,\text{rad}}^1(S_N)$, monotonically decreasing on each half--line and such that $F_{p,q}(u^*)<F_{p,q}(u)$.
	
	Therefore, let $(u_n)\subset H_{\mu,\text{rad}}^1(S_N)$ be a minimizing sequence for \eqref{radinf} and, due to Lemma \ref{propervar}, assume without loss of generality that $\|u_n\|_\infty=u_n(0)$ and $u_n$ is non--increasing on each half--line. By Gagliardo--Nirenberg inequalities \eqref{GNp}--\eqref{GNinf} it follows that $(u_n)$ is bounded in $H^1(S_N)$, so that $u_n\rightharpoonup u$ in $H^1(S_N)$ and $u_n\to u$ in $L_{\text{loc}}^\infty(S_N)$, for some $u\in H^1(S_N)$.
	
	Assume by contradiction that $u\equiv0$ on $S_N$. Then $u_n\to0$ in $L_{\text{loc}}^\infty(S_N)$, that is $u_n\to0$ in $L^\infty(S_N)$ since $u_n$ attains its $L^\infty$ norm at the origin. Thus $u_n\to0$ strongly in $L^p(S_N)$ and by weak lower semicontinuity
	\[
	\F_{p,q}^{\text{rad}}=\lim_n F_{p,q}(u_n)\geq0
	\]
	which is impossible by \eqref{neginf}. Thus $u\not\equiv0$ on $S_N$. 
	
	Let then $m:=\|u\|_2^2$, so that $\|u_n-u\|_2^2\to\mu-m$ as $n\to+\infty$ by weak convergence of $u_n$ to $u$ in $L^2(S_N)$, and assume by contradiction that $0<m<\mu$. 
	
	Since $u_n\in H_{\mu,\text{rad}}^1(S_N)$ is non--increasing on each half--line, then for every $n$ we have that $u_n-u$ is symmetric with respect to the origin and $(u_n-u)(0)\to0$ as $n\to+\infty$ since $u_n\to u$ in $L_{\text{loc}}^\infty(S_N)$. Hence, one can argue again as in the first part of the proof of Lemma \ref{GSsym} to construct $(u_n-u)^*\in H^1(S_N)$ symmetric with respect to the origin, non--increasing on each half--line and such that $\|(u_n-u)^*\|_2=\|u_n-u\|_2$ and $F_{p,q}((u_n-u)^*)\leq F_{p,q}(u_n-u)$.
	
	Thus, following the steps in the proof of Proposition \ref{compactth}, we have
	\[
	\begin{split}
	\F_{p,q}^{\text{rad}}(\mu)&\leq F_{p,q}\left(\sqrt{\f\mu{\|(u_n-u)^*\|_2^2}}\,(u_n-u)^*\right)\\
	&<\f{\mu}{\|(u_n-u)^*\|_2^2}F_{p,q}((u_n-u)^*)\leq\f{\mu}{\|u_n-u\|_2^2}F_{p,q}(u_n-u)\,,
	\end{split}
	\]
	that is
	\[
	\liminf_n F_{p,q}(u_n-u)\geq\f{\mu-m}\mu\F_{p,q}^{\text{rad}}(\mu)\,.
	\]
	Similarly, since $u\in H_{m,\text{rad}}^1(S_N)$,
	\[
	F_{p,q}(u)>\f m\mu \F_{p,q}^{\text{rad}}(\mu)\,.
	\]
	Relying again on $u_n'\rightharpoonup u'$ in $L^2(S_N)$, $u_n\to u$ in $L_{\text{loc}}^\infty(S_N)$ and on the Brezis--Lieb lemma \cite{BL} as in the proof of Proposition \ref{compactth}, we obtain
	\[
	\begin{split}
	\F_{p,q}^{\text{rad}}(\mu)=&\lim_n F_{p,q}(u_n)=\lim_n F_{p,q}(u_n-u)+F_{p,q}(u)\\
	>&\f{\mu-m}{\mu}\F_{p,q}^{\text{rad}}(\mu)+\f m\mu\F_{p,q}^{\text{rad}}(\mu)=\F_{p,q}^{\text{rad}}(\mu)\,,
	\end{split}
	\] 
	i.e. a contradiction. Henceforth, $u\in H_{\mu,\text{rad}}^1(S_N)$ and by weak lower semicontinuity $F_{p,q}(u)=\F_{p,q}^{\text{rad}}(\mu)$. In particular, $u$ is solution of \eqref{EulLag} for some $\omega>0$, so that it must coincide with the unique solution $\eta_0^{\omega(\mu)}$ to \eqref{EulLag} that verifies $\|\eta_0^{\omega(\mu)}\|_2^2=\mu$. 
\end{proof}

\section{Proof of Theorem \ref{exoutdiag}--\ref{exdiag} and of Proposition \ref{prop_Np}}
\label{sec:main}
This section is devoted to the details of the proof of the main results of the paper.

We begin with the following preliminary lemma.

\begin{lemma}
	\label{lem_crit}
	Let $\overline{\mu}>0$ be given. 
	\begin{itemize}
		\item[(i)] If $q<\f p2+1$ and ground states of $F_{p,q}$ exist at mass $\overline{\mu}$, then, for every $\mu\leq\overline{\mu}$, ground states at mass $\mu$ exist too.
		\item[(ii)] if $q>\f p2+1$ and ground states of $F_{p,q}$ exist at mass $\overline{\mu}$, then, for every $\mu\geq\overline{\mu}$, ground states at mass $\mu$ exist too.
	\end{itemize} 
\end{lemma}
\begin{proof}
	Let us start by proving statement \emph{(i)}. Suppose then that $q<\f p2+1$ and assume by contradiction that there exists $\widetilde{\mu}<\overline{\mu}$ such that ground states of $F_{p,q}$ at mass $\widetilde{\mu}$ do not exist. 
	
	Recall that, by Proposition \ref{GSsym}, whenever ground states exist they coincide with the unique stationary state $\eta_0^{\omega(\mu)}$ such that $\|\eta_0^{\omega(\mu)}\|_2^2$ coincides with the prescribed mass, so that in this case $\F_{p,q}(\mu)=F_{p,q}\left(\eta_0^{\omega(\mu)}\right)$. Therefore, relying also on Proposition \ref{compactth} and Corollary \ref{compactcor}, we have that ground states of $F_{p,q}$ at mass $\mu$ exist if and only if
	\[
	F_{p,q}\left(\eta_0^{\omega(\mu)}\right)\leq\Eps(\mu)
	\]
	which by \eqref{E phi mu} may be rewritten as
	\[
	\f{\F_{p,q}^{\text{rad}}(\mu)}{\mu^{2\beta+1}}\leq-\theta_p\,,
	\]
	where, as defined in \eqref{phi mu}, $\beta=\f{p-2}{6-p}$ and $\F_{p,q}^{\text{rad}}(\mu)=F_{p,q}\left(\eta_0^{\omega(\mu)}\right)$ by Proposition \ref{prop_radial}.
	
	Set $K(\mu):=\f{\F_{p,q}^{\text{rad}}(\mu)}{\mu^{2\beta+1}}$ for every $\mu>0$. Since $\F_{p,q}^{\text{rad}}(\mu)$ is a differentiable function of $\mu$ by Proposition \ref{prop_radial} and formula \eqref{Fpqstationary}, it follows that $K'$ exists for every $\mu>0$ and it verifies
	\begin{equation}
	\label{K'}
	K'(\mu)=\f1{\mu^{2\beta+1}}\left(\left(\F_{p,q}^{\text{rad}}\right)'(\mu)-(2\beta+1)\f{\F_{p,q}^{\text{rad}}(\mu)}\mu\right)\,,
	\end{equation}
	where $\left(\F_{p,q}^{\text{rad}}\right)'$ denotes the derivative of $\F_{p,q}^{\text{rad}}$.
	
	Let us now prove statement \emph{(i)}. Assume $q<\f p2+1$. Letting $\eta_0^{\omega(\mu)}$ be the radial, monotonically decreasing ground state at mass $\mu$, we have
	\[
	\begin{split}
	\F_{p,q}^{\text{rad}}(\mu-\varepsilon)-&\F_{p,q}^{\text{rad}}(\mu)\leq F_{p,q}
	\left(\sqrt{\f{\mu-\varepsilon}\mu}\,\eta_0^{\omega(\mu)}\right)-F_{p,q}\left(\eta_0^{\omega(\mu)}\right)\\
	=&\f12\f\varepsilon\mu\left(-\|\left(\eta_0^{\omega(\mu)}\right)'\|_2^2+\|\eta_0^{\omega(\mu)}\|_p^p+|\eta_0^{\omega(\mu)}(0)|^q\right)+o(\varepsilon)=\f12\varepsilon\omega(\mu)+o(\varepsilon)\,,
	\end{split}
	\]
	where $\varepsilon>0$ is sufficiently small. Therefore
	\[
	\left(\F_{p,q}^{\text{rad}}\right)'(\mu^-)\geq-\f12\omega(\mu)\qquad\forall\mu>0\,,
	\]
	where $\left(\F_{p,q}^{\text{rad}}\right)'(\mu^-)$ denotes the left derivative of $\F_{p,q}^{\text{rad}}$ at $\mu$.
	
	Conversely, denoting by $\left(\F_{p,q}^{\text{rad}}\right)'(\mu^+)$ the right derivative of $\F_{p,q}^{\text{rad}}$ at $\mu$, the same argument leads to 
	\[
	\left(\F_{p,q}^{\text{rad}}\right)'(\mu^+)\leq-\f12\omega(\mu)\,,
	\]
	so that, since $\F_{p,q}^{\text{rad}}$ is differentiable at every $\mu>0$,
	\[
	\left(\F_{p,q}^{\text{rad}}\right)'(\mu)=-\f12\omega(\mu)\,.
	\]
	Coupling with \eqref{K'} and making use of the explicit expression of $F_{p,q}\left(\eta_0^{\omega(\mu)}\right)$ as in \eqref{Fpqstationary} yields
	\[
	K'(\mu)=\f{p+2-2q}{q(6-p)}\f{|\eta_0(0)|^q}{\mu^{2\beta+2}}>0
	\]
	for every $\mu>0$, and \emph{(i)} is proved.
	
	The proof of statement \emph{(ii)} is analogous.
\end{proof}

\begin{proof}[Proof of Theorem \ref{exoutdiag}]
Set
\[
\mu_{p,q}:=\begin{cases}
\sup\{\mu>0\,:\,\F_{p,q}(\mu)\leq\Eps(\mu)\} & \text{if }q<\f p2+1\\
\inf\{\mu>0\,:\,\F_{p,q}(\mu)\leq\Eps(\mu)\} & \text{if }q>\f p2+1\,.
\end{cases}
\]
By Proposition \ref{compactth}, Corollary \ref{compactcor} and Lemma \ref{lem_crit}, it follows that if $q<\f p2+1$, then ground states of $F_{p,q}$ at mass $\mu$ exist if and only if $\mu\leq\mu_{p,q}$, whereas if $q>\f p2+1$, then ground states $F_{p,q}$ at mass $\mu$ exist if and only if $\mu\geq\mu_{p,q}$. Furthermore, Proposition \ref{prop_radial} ensures that, whenever they exist, ground states at prescribed mass are also unique. Thus, to complete the proof of Theorem \ref{exoutdiag} it is enough to show that
\[
0<\mu_{p,q}<+\infty\,.
\]

The proof is divided in two steps.

\medskip
\textit{Step 1. Existence.}
Let $u=(u_i)_{i=1}^N\subset H^{1}_{\mu}(S_{N})$ be given by
\begin{equation*}
u_i(x)=Ae^{-Bx}\quad\quad\text{on}\,\HH_{i}\,,\, i=1,\dots,N.
\end{equation*}
Imposing the boundary condition in \eqref{EulLag}, we get that 
  \begin{equation}
  \label{vertcondexp}
  NB=A^{q-2}.
  \end{equation}
  Furthermore,
  \begin{equation}
  \label{muexp}
  \mu=\frac{NA^2}{2B}
  \end{equation}
   and
   \begin{equation}
   \label{Fpqexp}
   F_{p,q}(u)=\frac{N}{4}A^2B-\frac{N}{2p}\frac{A^p}{B}-\frac{1}{q}A^q.
   \end{equation} 
   Combining \eqref{vertcondexp}, \eqref{muexp} and \eqref{Fpqexp}, we get
   \begin{equation*}
   \label{energyexp}
   F_{p,q}(u)=-\left(\frac{1}{q}-\frac{1}{4}\right)\left(\frac{2}{N^2}\right)^{\frac{q}{4-q}}\mu^{\frac{q}{4-q}}-\frac{N^{2}}{p^{2}}\left(\frac{2}{N^2}\right)^{\frac{p-q+2}{4-q}}\mu^{\frac{p-q+2}{4-q}}.
   \end{equation*}
   Now, if $q<\frac{p}{2}+1$, then 
   \begin{equation*}
   \frac{q}{4-q}<\frac{2q}{6-p}<\frac{p+2}{6-p}=2\beta+1
   \end{equation*}
   so that, recalling \eqref{E phi mu},
   \begin{equation*}
   \F_{p,q}(\mu)\le F_{p,q}(u)<\Eps(\mu) \qquad\text{for }\mu\text{ small enough,}
   \end{equation*}
   which shows that $\mu_{p,q}>0$.
   
   On the contrary, if $q>\frac{p}{2}+1$, then
   \begin{equation*}
   \frac{q}{4-q}>\frac{2q}{6-p}>\frac{p+2}{6-p}=2\beta+1
   \end{equation*}
   and consequently
   \begin{equation*}
   \F_{p,q}(\mu)\le F_{p,q}(u)<\Eps(\mu) \qquad\text{for }\mu\text{ large enough},
   \end{equation*}
   i.e. $\mu_{p,q}<+\infty$.

\medskip
 \textit{Step 2. Non-existence.}
As Proposition \ref{GSsym} ensures that if a ground state of \eqref{Fpq1} at mass $\mu$ exists, then it coincides with $\eta_0^{\omega(\mu)}$, relations \eqref{ftJ} and \eqref{mass1} become respectively
\begin{equation}
\label{tN}
\frac{t}{(1-t^2)^{\frac{q-2}{p-2}}}=\frac{\left(\frac{p}{2}\right)^{\frac{q-2}{p-2}}\omega(\mu)^{\frac{2q-2-p}{2(p-2)}}}{N}
\end{equation}
and
\begin{equation}
\label{muN}
\mu=2N\frac{\left(\frac{p}{2}\right)^\frac{2}{p-2}\omega(\mu)^{\frac{6-p}{2(p-2)}}}{p-2}\int_{t}^1(1-s^2)^{\frac{4-p}{p-2}}\,ds.
\end{equation}
Assume now $q<\frac{p}{2}+1$. Since
\begin{equation*}
0\le\int_{t}^1(1-s^2)^{\frac{4-p}{p-2}}\le \int_{0}^1(1-s^2)^{\frac{4-p}{p-2}}<+\infty \quad\text{for}\quad 2<p<6,
\end{equation*}   
by \eqref{muN} we get $\omega(\mu)\to \infty$ as $\mu\to+\infty$, and consequently $t\rightarrow 0^{+}$ by \eqref{tN}. Hence it follows that
   \begin{equation*}
   \omega(\mu)\sim \mu^{\frac{2(p-2)}{6-p}}\quad\text{as}\quad \mu\to+\infty.
   \end{equation*}
Let then $m:=\nlc{\phi_{\omega(\mu)}}{2}{\R}$, where $\phi_{\omega(\mu)}$ is the soliton \eqref{soliton} associated to the Lagrange multiplier $\omega(\mu)$. Since $t\to0$ as $\mu\to+\infty$, recalling \eqref{mass phi w} shows that
    \begin{equation*}
    \lim_{\mu\to+\infty}\frac{\mu}{m}=\frac{N}{2}
    \end{equation*}
and combining with \eqref{phi mu} gives
\begin{equation}
	\label{Linf gs}
	\left|\eta_0^{\omega(\mu)}(0)\right|^q\leq\|\phi_{\omega(\mu)}\|_{L^\infty(\R)}^q=C_p m^{\f{2q}{6-p}}\sim\mu^{\f{2q}{6-p}}\qquad\text{as }\mu\to+\infty\,.
\end{equation}
Conversely, since almost every value in the range of $\eta_0^{\omega(\mu)}$ is attained $N$ times on $S_N$, by \cite[Lemma 2.1]{ASTbound}
    \begin{equation}
    \label{E eta0}
  E\left(\eta_0^{\omega(\mu)},S_N\right)\geq -\theta_p\left(\frac{2}{N}\right)^{\frac{2(p-2)}{6-p}} \mu^{\frac{p+2}{6-p}} 
    \end{equation} 
    for every $\mu>0$. 
    
    Combining with \eqref{Linf gs} and the fact that $q<\f p2+1$ entails
    \[
    F_{p,q}\left(\eta_0^{\omega(\mu)}\right)\geq-\theta_p\left(\f2N\right)^{\frac{2(p-2)}{6-p}} \mu^{\frac{p+2}{6-p}}-C\mu^{\f{2q}{6-p}}\sim -\theta_p\left(\f2N\right)^{\frac{2(p-2)}{6-p}} \mu^{\frac{p+2}{6-p}}>\Eps(\mu)
    \]
    for $\mu$ sufficiently large and $N\geq3$, i.e. $\mu_{p,q}<+\infty$.
 
	If on the contrary $q>\frac{p}{2}+1$, then a similar argument shows that
 \begin{equation*}
      \omega(\mu)\sim \mu^{\frac{2(p-2)}{6-p}}\quad\text{as}\quad \mu\to0^{+}
 \end{equation*}
 and
 \begin{equation*}
 \left|\eta_0^{\omega(\mu)}(0)\right|^{q}\leq C\mu^{\frac{2q}{6-p}}
 \end{equation*}
 for $\mu$ small enough.
 
 Coupling again with \eqref{E eta0} then leads to
 \[
 F_{p,q}\left(\eta_0^{\omega(\mu)}\right)>\Eps(\mu)
 \]
 provided $\mu$ is sufficiently small, yielding $\mu_{p,q}>0$ and we conclude.
   \end{proof}
   
In the final part of the section we prove the main result in the case $q=\f p2+1$.
   
    \begin{proof}[Proof of Theorem \ref{exdiag}]
    If $q=\frac{p}{2}+1$, then \eqref{tN} reduces to
\begin{equation}
\label{tdiag}
 t=\frac{\sqrt{p}}{\sqrt{p+2N^2}}
\end{equation}
and \eqref{muN} can be rewritten as
\begin{equation*}
\mu=2N\frac{\left(\frac{p}{2}\right)^\frac{2}{p-2}\omega^{\frac{6-p}{2(p-2)}}}{p-2}\I(t),
\end{equation*}
where 
\begin{equation*}
\I(x):=\int_x^1(1-s^2)^{\frac{4-p}{p-2}}\,ds\qquad\forall x\in[0,1].
\end{equation*}
Then
\begin{equation}
\label{wdiag}
\omega(\mu)=\left(\frac{(p-2)\mu}{2N\I(t)\left(\frac{p}{2}\right)^{\frac{2}{p-2}}}\right)^{\frac{2(p-2)}{6-p}}\,,
\end{equation}
and plugging \eqref{tdiag} and \eqref{wdiag} into \eqref{Fpqstationary} for $\eta_0^{\omega(\mu)}$ gives
\begin{equation*}
F_{p,q}\left(\eta_0^{\omega(\mu)}\right)=-\frac{2^{\frac{p-2}{6-p}}(6-p)}{(p+2)p^{\frac{4}{6-p}}(N\I(t))^{\frac{2(p-2)}{6-p}}}\mu^{\frac{p+2}{6-p}}\,,
\end{equation*}
whereas making use of \eqref{phi mu} one can rewrite $\Eps(\mu)$ as
\begin{equation*}
\Eps(\mu)=-\frac{2^{\frac{p-2}{6-p}}(6-p)}{(p+2)p^{\frac{4}{6-p}}}\left(\frac{p-2}{\I(0)}\right)^{\frac{2(p-2)}{6-p}}\mu^{\frac{p+2}{6-p}}.
\end{equation*}
By Proposition \ref{compactth}, Corollary \ref{compactcor} and Proposition \ref{prop_radial}, ground states of \eqref{Fpq1} at mass $\mu$ exist if and only if 
\[
F_{p,q}\left(\eta_0^{\omega(\mu)}\right)\leq\Eps(\mu)
\]
that, thanks to the previous expressions, can be reduced to
\begin{equation}
\label{condcrit}
N\frac{\I\left(\sqrt{\frac{p}{p+2N^2}}\right)}{\I(0)}\le 2\,.
\end{equation}
Note that the function $h:\R^{+}\to\R^{+}$ defined as
\begin{equation}
\label{defh}
h(x):=x\int_{\sqrt{\frac{p}{p+2x^2}}}^1(1-s^2)^{\frac{4-p}{p-2}}\,ds
\end{equation}
is non-decreasing on $\R^{+}$. Indeed, differentiating \eqref{defh}, we have
\begin{equation*}
h'(x):=\int_{\sqrt{\frac{p}{p+2x^2}}}^1(1-s^2)^{\frac{4-p}{p-2}}\,ds+\frac{2^{\frac{2}{p-2}}\sqrt{p}x^\frac{4}{p-2}}{(p+2x^2)^{\frac{p+2}{2(p-2)}}}>0\qquad\forall x>0\,.
\end{equation*} 
Therefore, observing that the left-hand side of \eqref{condcrit} is strictly less than $2$ for $N=2$, it is non-decreasing in $N$ and diverges to $+\infty$ as $N\to+\infty$, since $\I$ is continuous and bounded and $\f{p}{p+2N^2}\to0$, then there exists $N_p\ge2$ such that existence of ground states is guaranteed for $N\le N_p$, while non-existence holds for $N>N_p$.\\
\end{proof}

\begin{figure}[t]
	\centering
	\includegraphics[width=1\textwidth]{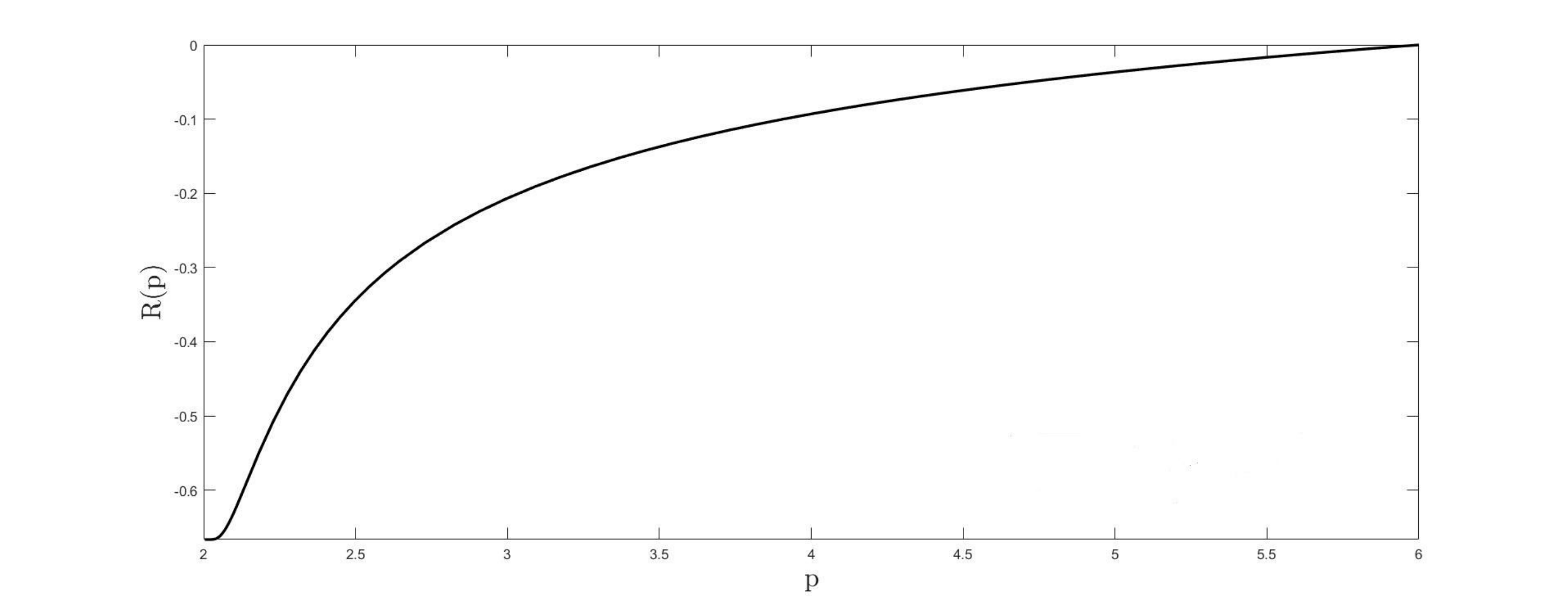}
	\caption{The graph of $R(p):=\f{\mathcal{I}\left(\sqrt{\f p{p+18}}\right)}{\mathcal{I}(0)}$ as a function of $p\in(2,6)$. The validity of \eqref{condcrit} at $N=3$ is equivalent to $R(p)\leq0$. These numerical simulations suggest that condition \eqref{condcrit} always holds at $N=3$, for every $p\in(2,6)$. }
	\label{fig_sim}
\end{figure}

As already pointed out in the Introduction and displayed clearly in the previous proof, without further assumption on $p\in(2,6)$, at the moment we can prove only that $N_p$ in Theorem \ref{exdiag} satisfies $N_p\geq2$ for every $p\in(2,6)$. However, numerical simulations (see Figure \ref{fig_sim}) strongly suggest that actually $N_p\geq3$ for every $p\in(2,6)$. To conclude, we thus provide the proof of Proposition \ref{prop_Np}.

\begin{proof}[Proof of Proposition \ref{prop_Np}]
	We split the proof in two parts.
	
	\smallskip\textit{Statement (i).} Fix $N>0$. We need to show that $N_p\geq N$ as soon as $p$ is sufficiently close to 2. To this end, we verify that, given $N$, condition \eqref{condcrit} is always satisfied when $p$ approaches $2$. 
	
	Note that if $p\in(2,4)$, then $(1-s^2)^{\f{4-p}{p-2}}$ is a decreasing function of $s\in[0,1]$. Hence, 
	\[
	\mathcal{I}\left(\sqrt{\f p{p+2N^2}}\right)=\int_{\sqrt{\f p{p+2N^2}}}^1 (1-s^2)^{\f{4-p}{p-2}}\,ds\leq\left(\f{2N^2}{p+2N^2}\right)^{\f{4-p}{p-2}}\left(1-\sqrt{\f p{p+2N^2}}\right)\,.
	\]
	Conversely, 
	\[
	\mathcal{I}(0)=\int_0^1(1-s^2)^{\f{4-p}{p-2}}\,ds=\int_0^1(1+s)^{\f{4-p}{p-2}}(1-s)^{\f{4-p}{p-2}}\,ds\geq\int_0^1(1-s)^{\f{4-p}{p-2}}\,ds=\f{p-2}2\,.
	\]
	Therefore we get
	\[
	N\frac{\I\left(\sqrt{\frac{p}{p+2N^2}}\right)}{\I(0)}\leq\f {2N}{p-2}\left(\f{2N^2}{p+2N^2}\right)^{\f{4-p}{p-2}}\left(1-\sqrt{\f p{p+2N^2}}\right)\to0\qquad\text{as }p\to2^+\,,
	\]
	so that \eqref{condcrit} holds true as soon as $p$ is close enough to 2. This shows that there exists $\delta=\delta(N)$ so that for every $p\in(2,2+\delta)$ we have $N_p\geq N$, i.e. Proposition \ref{prop_Np}\emph{(i)} is proved.
	
	\smallskip\textit{Statement (ii).} Let $p\in[4,6)$. We begin by showing that $N_p<5$, proving that \eqref{condcrit} fails whenever $p\in[4,6)$ and $N\geq5$.
	
	Note that assuming $p\geq4$ implies that $(1-s^2)^{\f{4-p}{p-2}}$ is increasing as a function of $s$ and that $(1-s^2)^{\f{4-p}{p-2}}\geq1$ on $[0,1]$. Henceforth, we get
	\[
	\mathcal{I}\left(\sqrt{\f p{p+2N^2}}\right)\geq1-\sqrt{\f{p}{p+2N^2}}\,.
	\]
	Moreover, fixing $s\in[0,1)$ and regarding $(1-s^2)^{\f{4-p}{p-2}}$ as a function of $p\in[4,6)$, we have
	\[
	\f{d}{dp}(1-s^2)^{\f{4-p}{p-2}}=-(1-s^2)^{\f{4-p}{p-2}}\ln(1-s^2)\f{2}{(p-2)^2}>0\,,
	\]
	i.e. for every given $s\in(0,1)$, $(1-s^2)^{\f{4-p}{p-2}}$ is an increasing function of $p$. Therefore, for every $s\in[0,1)$ and $p\in[4,6)$
	\[
	(1-s^2)^{\f{4-p}{p-2}}\leq(1-s^2)^{-\f12}\,,
	\]
	so that integrating over $[0,1]$ gives
	\[
	\mathcal{I}(0)\leq\int_0^1(1-s^2)^{-\f12}\,ds=\f\pi2\,.
	\]
	We thus obtain
	\[
	N\f{\mathcal{I}\left(\sqrt{\f{p}{p+2N^2}}\right)}{\mathcal{I}(0)}\geq\f{2N}\pi\left(1-\sqrt{\f{p}{p+2N^2}}\right)\geq\f{2N}\pi\left(1-\sqrt{\f{3}{3+N^2}}\right)\,,
	\]
	where the last inequality follows from the fact that, for every given $N$, $\sqrt{\f{p}{p+2N^2}}$ is an increasing function of $p$. 
	
	In view of \eqref{condcrit}, to prove that $N_p\in\{2,3,4\}$ for every $p\in[4,6)$, it is then enough to show that, for every $N\geq5$
	\[
	\f{2N}\pi\left(1-\sqrt{\f{3}{3+N^2}}\right)>2\,.
	\]
	that can be equivalently rewritten as
	\begin{equation}
	\label{eq:G(N)}
	G(N):=N^4-2\pi N^3+\pi^2 N^2-6\pi N+3\pi^2>0\,.
	\end{equation}
	To prove that \eqref{eq:G(N)} holds for every $N\geq5$, we will show that it is true when $N=5$ and that $G(N)$ is increasing function of $N$ on $[5,+\infty)$.
	
	On the one hand, direct calculations immediately show that $G(5)>0$. On the other hand, differentiating \eqref{eq:G(N)} with respect to $N$ gives
	\[
	G'(N)=4N^3-6\pi N^2+2\pi^2 N-6\pi\,,
	\]
	and again we directly see that $G'(5)>0$. A further differentiation leads to
	\[
	G''(N)=12N^2-12\pi N+2\pi^2\,,
	\]
	yielding
	\[
	G''(N)>0\qquad\text{for every }N\geq\f{3+\sqrt{3}}{6}\pi\,.
	\]
	Since $5>\f{3+\sqrt{3}}{6}\pi$ and $G'(5)>0$, this implies that $G'(N)\geq0$ for every $N\geq5$, and coupling with $G(5)>0$ this ensures that $G(N)\geq0$ for every $N\geq5$. This concludes the proof of the first part of Proposition \ref{prop_Np}\emph{(ii)}.
	
	We are then left to show that there exists $\delta>0$ such that, if $p\in(4-\delta,4+\delta)\cup(6-\delta,6)$, then $N_p=3$. 
	
	We start by proving the results for $p$ in a suitable neighbourhood of 4. Note that, for every given $N$, we have
	\[
	N\f{\mathcal{I}\left(\sqrt{\f 4{4+2N^2}}\right)}{\mathcal{I}(0)}=N\f{\int_{\sqrt{\f 4{4+2N^2}}}^1\,ds}{\int_0^1\,ds}=N\left(1-\sqrt{\f 2{2+N^2}}\right)\,.
	\]
	Hence, evaluating the previous expression at $N=3$ and $N=4$ respectively, we obtain
	\[
	3\left(1-\sqrt{\f2 {11}}\right)<2\,,\qquad4\left(1-\sqrt{\f 2{18}}\right)=\f 83>2\,,
	\]
	so that condition \eqref{condcrit} is satisfied at $p=4,\,N=3$, whereas it fails at $p=4,\,N=4$, in turn implying $N_4=3$. Since both previous inequalities are strict, by continuity with respect to $p$, we conclude that $N_p=3$ for every $p\in(4-\delta,4+\delta)$, for some $\delta>0$.
	
	Let us now concentrate on the case $p\in(6-\delta,6)$. When $p=6$, evaluating the left hand side of \eqref{condcrit} at $N=3$ and $N=4$ gives
	\[
	\begin{split}
	3\f{\int_{\f12}^1(1-s^2)^{-\f12}\,ds}{\int_0^1(1-s^2)^{-\f12}\,ds}=&3\f{\f\pi2-\f\pi6}{\f\pi2}=2\,,\\
	4\f{\int_{\sqrt{\f3{19}}}^1(1-s^2)^{-\f12}\,ds}{\int_0^1(1-s^2)^{-\f12}\,ds}=&4\f{\f\pi 2-\arcsin\left(\sqrt{\f3{19}}\right)}{\f\pi2}>2\,.
	\end{split}
	\]
	On the one hand, the second inequality being strict shows that \eqref{condcrit} is violated when $N=4$ and $p\in(6-\delta,6)$, for suitable $\delta>0$, so that for all these exponents it must be $N_p\leq3$. 
	
	On the other hand, the first line of the previous equation shows that \eqref{condcrit} becomes an equality when $p=6$ and $N=3$. Therefore, to show that $N_p=3$, we need to further analyse the behaviour of the left hand side of \eqref{condcrit} when $N=3$ and $p$ approaches 6. To do this, fix $N=3$ and set
	\begin{equation}
	\label{R(p)}
	R(p):=\f{\mathcal{I}\left(\sqrt{\f p{p+18}}\right)}{\mathcal{I}(0)}\,.
	\end{equation}
	We will conclude the proof by showing that $R'(6)>0$. By continuity, this eventually guarantees the existence of $\delta>0$ so that $N_p=3$ for every $p\in(6-\delta,6)$.
	
 	Differentiating \eqref{R(p)} with respect to $p$ we obtain
	\begin{equation}
	\label{R'(6)}
	\begin{split}
	R'(p)=&\f{\left(-\f{9\sqrt{p+18}}{\sqrt{p}(p+18)^2}\left(\f{18}{p+18}\right)^{\f{4-p}{p-2}}-\f{2}{(p-2)^2}
	\int_{\sqrt{\f p{p+18}}}^{1}(1-s^2)^{\f{4-p}{p-2}}\ln(1-s^2)\,ds\right)\int_0^1(1-s^2)^{\f{4-p}{p-2}}}{\left(\int_0^1(1-s^2)^{\f{4-p}{p-2}}\right)^2}\\
	+&\f{\f2{(p-2)^2}\int_{\sqrt{\f{p}{p+18}}}^1(1-s^2)^{\f{4-p}{p-2}}\,ds\int_0^1(1-s^2)^{\f{4-p}{p-2}}\ln(1-s^2)\,ds}{\left(\int_0^1(1-s^2)^{\f{4-p}{p-2}}\right)^2}\,.
	\end{split}
	\end{equation}
	According to the previous expression, the numerator of $R'(6)$ reads
	\[
	\left(-\f{1}{16\sqrt{3}}-\f1 8\int_{\f12}^{1}(1-s^2)^{-\f12}\ln(1-s^2)\,ds\right)\f\pi2+\f18\left(\f\pi2-\f\pi6\right)\int_0^1(1-s^2)^{-\f12}\ln(1-s^2)\,ds\,,
	\]
	that can be rewritten as
	\begin{equation}
	\label{numR'(6)}
	\f{\pi}{16}\left(-\f13\int_0^1(1-s^2)^{-\f12}\ln(1-s^2)\,ds-\f1{2\sqrt{3}}+\int_0^{\f12}(1-s^2)^{-\f12}\ln(1-s^2)\,ds\right)\,.
	\end{equation}
	The first integral in the above bracket can be computed explicitly making use of polylogarithmic functions
	\begin{equation}
	\label{piln2}
	\int_0^1(1-s^2)^{-\f12}\ln(1-s^2)\,ds=-\pi\ln2\,.
	\end{equation}
	Furthermore, since $(1-s^2)^{-\f12}\ln(1-s^2)$ is a decreasing function of $s$ on $\left[0,\f12\right]$, we have
	\[
	\int_0^{\f12}(1-s^2)^{-\f12}\ln(1-s^2)\,ds\geq\int_0^{\f12}\f2{\sqrt{3}}\ln\left(\f34\right)\,ds=\f1{\sqrt{3}}\ln\left(\f34\right)\,.
	\]
	As it holds 
	\[
	\f{\pi\ln2}3>\f1{2\sqrt{3}}-\f1{\sqrt{3}}\ln\left(\f34\right)\,,
	\]
	combining with \eqref{R'(6)},\eqref{numR'(6)} and \eqref{piln2} we have $R'(6)>0$, and the proof is complete.
\end{proof}

\section{Proof of Proposition \ref{prop_stab}}
\label{sec:stab}
In this section we prove the orbital stability of the radial stationary state $\eta_0^{\omega(\mu)}$, even for values of the mass for which there is no ground state. The proof of Proposition \ref{prop_stab} strongly relies on a method introduced in \cite{acfn_jde}, which is essentially based on the reduction of an  infinite-dimensional problem to a finite-dimensional one. The argument of \cite{acfn_jde} extends almost straightforwardly to our setting. The idea of the method is the following. We map continuously every function in the mass constrained space into another function made of pieces of solitons, whose energy is lower than the one of the original function. Then, we prove that the radial stationary states are local minimum points for the energy among functions with the same mass and made of pieces of solitons. Thus, a fortiori, they are local minima in the whole mass constrained energy space, and therefore, due to the general stability theory, they are orbitally stable. Here we limit ourselves to sketch the main steps of the proof, explicitly pointing out the minor differences with respect to \cite{acfn_jde} whenever occurring.

\begin{proof}[Proof of Proposition \ref{prop_stab}]
	Owing to Theorem 3 in \cite{GSS}, the orbital stability of $\eta_0^{\omega(\mu)}$ is equivalent to its local minimality for $F_{p,q}$ in $\HmuS$, hence we prove that $\eta_0^{\omega(\mu)}$ is a local minimum for the energy in $\HmuS$.
	
	Following \cite[Definition 2.1]{acfn_jde}, we fix $\mu>0$ and define the multi-soliton manifold $\M$ as the subspace of $\HmuS$ made of all the functions whose restriction to each half-line $\mathcal{H}_{j}$ of $S_{N}$ gives a piece of soliton, i.e.
	\[
	\M:=\left\{u\in\HmuS\,:\,u_j=\phi_{\omega_j}(\cdot+a_j),\,\text{ for some }\omega_j,\,a_j\in\R,\,j=1,\,\dots,\,N\right\}\,.
	\]
	Given a function $\eta\in \HmuS$ such that $\eta(0)\neq 0$, we define the multi-soliton transformation of $\eta$ as the unique function $\Sigma\eta\in\M$ so that the restriction $(\Sigma\eta)_j$ of $\Sigma\eta$ to the half--line $\HH_{j}$ satisfies
	\begin{equation*}
	(\Sigma \eta)_{j}:=\phi_{\omega(m_{j}, h)}(\cdot+a(m_{j}, h)),
	\end{equation*}
	where 
	\[
	m_{j}:=\int_{\mathcal{H}_{j}} \eta_{j}^2\, dx_{j},\,\,j=1,\dots,N,\quad\text{ and }\quad h=|\eta(0)|,
	\]
	and $\phi_{\omega(m_{j}, h)}(\cdot+a(m_{j}, h))$ is the unique piece of soliton with mass $m_{j}$ and $\phi_{\omega(m_{j}, h)}(a(m_{j}, h))=h$. For every given values of $m_j,h$, the uniqueness of such $\phi_{\omega(m_{j}, h)}(\cdot+a(m_{j}, h))$ has been proved in \cite[Theorem 4.1]{ASTcpde}. Furthermore, by \cite[Remark 3.4]{acfn_jde} we have for every $\eta$ so that $\eta(0)\neq0$
	\[
	F_{p,q}(\Sigma\eta)\leq F_{p,q}(\eta)\,,
	\]
	and equality holds if and only if $\eta\in\M$, that is $\Sigma\eta\equiv\eta$. In light of this and of the continuity of $\Sigma$ \cite[Proposition 3.2]{acfn_jde}, to show that $\eta_0^{\omega(\mu)}$ is a local minimizer of $F_{p,q}$ in $\HmuS$, it is enough to prove that it locally minimizes the energy in $\M$. Note that $\eta_0^{\omega(\mu)}\in\M$ for every $\mu>0$.
	
	We observe that any given function $\varphi\in \M$ corresponds to a point $P=(m_{1},\dots,m_{N-1},h)\in (0,+\infty)^{N}$, where $m_{j}$ is the mass of the restriction $\varphi_j$ of $\varphi$ to $\mathcal{H}_{j}$, $j=1,\,\dots,\,N-1$, and $h=|\varphi(0)|$. Therefore, it is natural to define the reduced energy function $r:(0,+\infty)^{N}\to \R$ as
	\begin{equation*}
	r(P):=F_{p,q}(\varphi), 
	\end{equation*}
	which can be conveniently decomposed as follows
	\begin{equation*}
	r(P)=\sum_{i=1}^{N-1}e(m_{i},h)+ e\left(\mu-\sum_{i=1}^{N-1}m_{i},h\right),
	\end{equation*}
	where $e:(0,+\infty)\times \R^{+}\to \R$ is given by
	\begin{equation*}
	e(m,h):=\f{1}{2}\|\phi'_{\omega(m, h)}(\cdot+a(m,h))\|^{2}_{L^{2}(\R^{+})}-\f{1}{p}\|\phi_{\omega(m, h)}(\cdot+a(m,h))\|^{p}_{L^{p}(\R^{+})}-\f{1}{qN}h^{q}.
	\end{equation*}
	Thus, setting  $\bar{h}=|\eta_0^{\omega(\mu)}(0)|$, the local minimality of $\eta_0^{\omega(\mu)}$ in $\M$ is equivalent to the local minimality  for $r$ of the point $\overline{P}=\left(\f{\mu}{N},\dots, \f{\mu}{N},\bar{h}\right)$.
	
	Since $\overline{P}$ is an internal point of $(0,+\infty)^{N}$ and it is a stationary point for $r$ as $\eta_0^{\omega(\mu)}$ is a critical point for $F_{p,q}$, to conclude it is then sufficient to prove that the Hessian matrix of $r$ evaluated at $\overline{P}$ is positive definite. By straightforward computations,
	\begin{equation*}
	\begin{split}
	\f{\partial^{2} r}{\partial m_{i}\partial m_{j}}&(\overline{P})=(1+\delta_{ij})\f{\partial^{2} e}{\partial m^{2}}\left(\f{\mu}{N},\bar{h}\right),\\
	\f{\partial^{2} r}{\partial h^{2}}&(\overline{P})=N\f{\partial^{2} e}{\partial h^{2}}\left(\f{\mu}{N},\bar{h}\right),\\
	\f{\partial^{2} r}{\partial m_{i}\partial h}&(\overline{P})=0\,,
	\end{split}
	\end{equation*}
	where $\delta_{ij}$	denotes as usual the Kronecker's symbol of $i, j$. By elementary linear algebra, one easily sees that the Hessian matrix has three eigenvalues: $N\f{\partial^{2} e}{\partial m^{2}}\left(\f{\mu}{N},\bar{h}\right)$ with multiplicity $1$, $\f{\partial^{2} e}{\partial m^{2}}\left(\f{\mu}{N},\bar{h}\right)$ with multiplicity $N-2$ and $\f{\partial^{2} e}{\partial h^{2}}\left(\f{\mu}{N},\bar{h}\right)$ with multiplicity $1$. Therefore, to show that the Hessian matrix is positive definite, we need to prove that
	\begin{equation}
	\label{dFdm}
	\f{\partial^{2} e}{\partial m^{2}}\left(\f{\mu}{N},\bar{h}\right)>0
	\end{equation}
	and
	\begin{equation}
	\label{dFdh}
	\f{\partial^{2} e}{\partial h^{2}}\left(\f{\mu}{N},\bar{h}\right)>0.
	\end{equation}
	The proof of \eqref{dFdm}--\eqref{dFdh} is analogous to the one of inequalities (4.2)--(4.3) of \cite{acfn_jde}. The main idea is to consider the variations
	\begin{equation*}
	f_1(t):=r(\underline{m}(t),\bar{h})\quad\text{ and }\quad f_2(t):=r\left(\f\mu N,\,\dots,\f\mu N,\bar{h}+t\right),\quad t \in (-\varepsilon,\varepsilon)
	\end{equation*}
	where $\underline{m}(t)=\left(\f{\mu}{N}+t,\f{\mu}{N},\dots,\f{\mu}{N}\right)$. 
	
	It is plainly seen that
	\[
	f_1''(0)=2\f{\partial^{2} e}{\partial m^{2}}\left(\f{\mu}{N},\bar{h}\right)\quad\text{ and }\quad f_2''(0)=N\f{\partial^{2} e}{\partial h^{2}}\left(\f{\mu}{N},\bar{h}\right).
	\]
	Note that $f_1(t)$ corresponds to an exchange of mass $t$ between the first and the $N$--th half--lines, without involving the remaining $N-2$ ones. Hence, 
	\[
	f_1(t)-f_1(0)=f_1(t)-r\left(\overline{P}\right)=\widetilde{F}_{p,q}\left(\varphi_t,\R\right)-\widetilde{F}_{p,q}\left(\varphi_0,\R\right)
	\]
	where $\varphi_t,\varphi_0\in H^1_{\f{2\mu}N}(\R)$ denote respectively the restriction to the line $\HH_{1}\cup\HH_{N}$ of the function $\eta_t\in\M$ corresponding to the point $(\underline{m}(t),\bar{h})$ and of the stationary state $\eta_0^{\omega(\mu)}$, and $\widetilde{F}_{p,q}:H_{\f{2\mu}N}^1(\R)\to\R$ is given by
	\[
	\widetilde{F}\left(u,\R\right)=\f{1}{2}\|u'\|^{2}_{L^{2}(\R)}-\f{1}{p}\|u\|^{p}_{L^{p}(\R)}-\f{2}{qN}|u(0)|^{q}\,.
	\]
	Here is the point where we need to argue slightly differently with respect to \cite{acfn_jde}. Indeed, we now rely on the results of \cite{BD}, which guarantees that $\varphi_0$ as above is a global minimizer of $\widetilde{F}_{p,q}$ in $H_{\f{2\mu}N}^1(\R)$. Coupling with the stability result in \cite[Theorem 3.4]{GSS}, this immediately yields
	\[
	f_1(t)-f_1(0)\geq C\|\varphi_t-\varphi_0\|_{L^2(\R)}^2
	\]
	for some constant $C>0$ and $t$ small enough. Moving from the previous inequality, and repeating the same calculations as in \cite[pp. 7411]{acfn_jde}, we eventually obtain
	\[
	f_1(t)-f_1(0)\geq C t^2\,,
	\]
	that coupled with $f_1'(0)=0$ ensures $f_1''(0)\geq C>0$. This proves \eqref{dFdm}. Since the same argument developed for $f_2$ leads to \eqref{dFdh}, we conclude.
\end{proof}

\section*{Acknowledgements}
The first and the second authors acknowledge that the present research has been partially supported by MIUR grant Dipartimenti di Eccellenza 2018-2022 (E11G18000350001). All the authors wish to thank Enrico Serra and Paolo Tilli for fruitful discussions and suggestions.

\end{document}